\documentclass[10pt,reqno]{amsart}
\usepackage{enumitem}
\usepackage{amscd}
\usepackage{amssymb}
\usepackage[all]{xy}
\usepackage{fge}
\newcommand{\moins}{\mathbin{\fgebackslash}}

\usepackage[textsize=tiny]{todonotes}



\RequirePackage[T1]{fontenc}
\RequirePackage{amsfonts,latexsym,amssymb}

\let\cal\mathcal

\usepackage{url}
\usepackage{bbm}
\usepackage[pagebackref]{hyperref}

\newtheorem{theorem}[equation]{Theorem}
 
 \newtheorem{proposition}[equation]{Proposition}
 \newtheorem{corollary}[equation]{Corollary}  
\newtheorem{conjecture}[equation]{Conjecture}

\theoremstyle{definition}

\newtheorem{question}[equation]{Question}
\newtheorem{remark}[equation]{Remark}

\newtheorem{example}[equation]{Example}
\theoremstyle{remark}

\def\jcdot{\scriptscriptstyle\bullet}

\def\invlim{\mathop{\vtop{\ialign{##\crcr$\hfill{\lim}\hfil$\crcr
\noalign{\kern1pt\nointerlineskip}\leftarrowfill\crcr\noalign
{\kern -3pt}}}}\limits}
\def\dirlim{\mathop{\vtop{\ialign{##\crcr$\hfill{\lim}\hfil$\crcr
\noalign{\kern1pt\nointerlineskip}\rightarrowfill\crcr\noalign
{\kern -3pt}}}}\limits} 
\def\lomapr#1{\smash{\mathop{\relbar\joinrel\longrightarrow}\limits^{#1}}}

\def\phi{\varphi} 
\def\epsilon{\varepsilon}

\newcommand{\ovk}{\overline{K} }

\newcommand{\dr}{\operatorname{dR} }

      \newcommand{\nr}{\operatorname{nr} }   
 
 \newcommand{\colim}{\operatorname{colim} }

 \newcommand{\proeet}{\operatorname{pro\acute{e}t} } 
 \newcommand{\eet}{\operatorname{\acute{e}t} }

 \newcommand{\Hom}{{\rm{Hom}} }
  
  \newcommand{\Hhom}{{\cal Hom}} 
 \newcommand{\Ext}{\operatorname{Ext} }

\newcommand{\Gal}{\operatorname{Gal} }
\newcommand{\can}{ \operatorname{can} }
\newcommand{\synt}{ \operatorname{syn} }
 
\newcommand{\st}{\operatorname{st} }
 \newcommand{\kker}{\operatorname{Ker} }

 \newcommand{\crr}{\operatorname{cr} }
  \newcommand{\hk}{\operatorname{HK} }
   \newcommand{\LL}{\operatorname{L} }

 \newcommand{\so}{{\mathcal O}}

\newcommand{\sd}{{\mathcal{D}}}

\newcommand{\wotimes}{\widehat{\otimes}}
 
 \newcommand{\wh}{\widehat}
   \numberwithin{equation}{section}

\def\R{{\mathrm R}}
\def\O{{\cal O}}
  \def\B{{\bf B}}
\def\Q{{\bf Q}} \def\Z{{\bf Z}}
\def\C{{\bf C}}
\def\N{{\bf N}}
\def\O{{\cal O}}
 
\def\Qbar{\overline{\bf Q}}

\def\ainf{{\bf A}_{{\rm inf}}}
\def\bcris{{\bf B}_{{\rm cris}}} \def\acris{{\bf A}_{{\rm cris}}}
\def\bst{{\bf B}_{{\rm st}}}
\def\brig{{\bf B}_{{\rm rig}}}
\def\bdr{{\bf B}_{{\rm dR}}}
\def\bpdr{{\bf B}_{{\rm pdR}}}
\def\bpFF{{\bf B}_{{\rm pFF}}}
\def\bFF{{\bf B}_{{\rm FF}}}
\def\Bcris{{\mathbb B}_{{\rm cris}}} 
\def\Bdr{{\mathbb B}_{{\rm dR}}}
\def\Bst{{\mathbb B}_{{\rm st}}}

\def\rg{{\rm R}\Gamma}

\def\epsilon{\varepsilon}
\def\dual{{\boldsymbol *}}

 \makeatletter
\let\@wraptoccontribs\wraptoccontribs
\makeatother

\numberwithin{equation}{section}
\makeindex
\begin{document}
\title[Hodge Theory of  $p$-adic  analytic varieties: a survey]
 {Hodge Theory of  $p$-adic analytic varieties: a survey}
 \author{Pierre Colmez} 
\address{CNRS, IMJ-PRG, Sorbonne Universit\'e, 4 place Jussieu, 75005 Paris, France}
\email{pierre.colmez@imj-prg.fr} 
\author{Wies{\l}awa Nizio{\l}}
\address{CNRS, IMJ-PRG, Sorbonne Universit\'e, 4 place Jussieu, 75005 Paris, France}
\email{wieslawa.niziol@imj-prg.fr}
 \date{\today}
\thanks{P.C. and W.N.'s research was supported in part by the  Simons Collaboration on Perfection in Algebra, Geometry, and Topology}
\maketitle
 \begin{abstract}
Hodge Theory of $p$-adic analytic varieties was initiated by Tate in his 1967 paper 
 on
$p$-divisible groups, where he conjectured the existence of a Hodge-like decomposition for the
$p$-adic \'etale cohomology of proper analytic varieties.  
Tate's conjecture was refined
by Fontaine who gave the theory its definite shape.  A lot of work has been done
for algebraic varieties and a number of proofs of Fontaine's conjectures
have been obtained between years 1985 and 2011.  But the study of Hodge Theory
of $p$-adic analytic varieties started really only in 2011 with Scholze's proof
of Tate's conjecture using perfectoid methods. Methods  that  opened the way to
an avalanche of results.

In this paper, we survey
our results and conjectures (comparison theorems
and their geometrization, dualities, etc.), focusing on the case of nonproper analytic 
varieties, where a number of new phenomena occur.  
We also describe the new objects that appeared along the way.
\end{abstract}

\tableofcontents

\section{Introduction}
\subsection{Complex periods}
\subsubsection{${\cal C}^\infty$-manifolds}
If $M$ is a compact, connected, orientable ${\cal C}^\infty$-manifold of dimension~$d$, 
the classical theorem of de Rham states that
integrating a closed differential form along  a cycle without boundary induces
a pairing $$H^i_{{\rm dR},{\cal C}^\infty}(M)\times H_i(M,\Z)\to{\bf R}$$ between de Rham cohomology
and singular homology, which identifies
$H^i_{{\rm dR},{\cal C}^\infty}(M)$ with ${\rm Hom}(H_i(M,\Z),{\bf R})$.  This can be rephrased 
using Betti cohomology $H^i_{\rm B}(M,\Q)\simeq {\rm Hom}(H_i(M,\Z),\Q)$
by saying that we have a functorial  period isomorphism
\begin{equation}\label{derham1}
H^i_{{\rm dR},{\cal C}^\infty}(M)\simeq {\bf R}\otimes_\Q H^i_{\rm B}(M,\Q),
\end{equation}
a statements that can be proved, following Weil (see \cite{CW}, Weil 18/01/47, Weil 02/02/47), 
by means of the
${\cal C}^\infty$-Poincar\'e Lemma.

Moreover $H^d_{{\rm dR},{\cal C}^\infty}(M)\simeq{\bf R}$ and $H^d_{\rm B}(M)\simeq\Q$, and the
resulting
cup-products $H^i_{{\rm dR},{\cal C}^\infty}(M)\otimes H^{d-i}_{{\rm dR},{\cal C}^\infty}(M)\to
{\bf R}$ and $H^i_{\rm B}(M)\otimes H^{d-i}_{\rm B}(M)\to \Q$
are perfect dualities (Poincar\'e dualities).

The isomorphism~(\ref{derham1}) extends to manifolds which are countable at infinity but, since
the cohomology groups can now be infinite dimensional, one needs to take a completed tensor
product. In this noncompact situation we have at our disposal other cohomology theories,
using classes with compact support, and Poincar\'e duality is now a perfect duality
between cohomology and cohomology with compact support (see~\cite{weil}, for example).

\subsubsection{Complex algebraic and analytic varieties}
Now, if $M=X(\C)$ is the ${\cal C}^\infty$-manifold defined by the $\C$-points 
of a proper smooth algebraic variety $X$
defined over a subfield $K$ of $\C$, the de Rham cohomology of $M$ can be recovered~\cite{GdR} from the
cohomology of the algebraic de Rham complex
$$\rg_{\rm dR}(X):=\R\Gamma(X, \O\to\Omega^1\to\Omega^2\to\cdots)$$
 of $X$: we have functorial isomorphisms
$$\C\otimes_K H^i_{\rm dR}(X)\simeq \C\otimes_{\bf R} H^i_{{\rm dR},{\cal C}^\infty}(X(\C)).
$$  
Combined with the 
theorem of de Rham from \eqref{derham1}, this gives functorial isomorphisms 
\begin{equation}\label{ht2}
\C\otimes_KH^i_{\rm dR}(X)\simeq
\C\otimes_\Q H^i_{\rm B}(X(\C),\Q).
\end{equation}
  The coefficients of the matrix of this isomorphism
in bases over $K$ and $\Q$ respectively are called ``periods''; for example $2\pi i$ is a period
of ${\bf P}^1$ (for $H^2$) or of ${\bf G}_m$ (for $H^1$, but ${\bf G}_m$ is not proper),
and $\frac{\Gamma(1/4)\Gamma(1/2)}{\Gamma(3/4)}$
 is a period of the elliptic curve $y^2=x^3-x$ (for $H^1$).
See~\cite{KZ} for general conjectures about these numbers. 

The algebraic de Rham cohomology is endowed with a natural filtration (the Hodge filtration)
given by the  filtration of the de Rham complex:
$${\rm Fil}^i\rg_{\rm dR}(X):=\rg(X,0\to\cdots\to 0\to\Omega^i\to\Omega^{i+1}\to\cdots)$$
 Hodge theory provides a canonical splitting of this filtration over $\C$ by means of
harmonic forms of type $(p,q)$, which gives a functorial (Hodge) decomposition
\begin{equation}\label{ht1}
\C\otimes_\Q H^i_{\rm B}(X(\C),\Q)\simeq\oplus_{p+q=i}\C\otimes_K H^p(X,\Omega^q)
\end{equation}
This decomposition does not, in general, come from a decomposition over $K$.

Most of the above extends to compact complex analytic varieties (taking $K=\C$), but the existence of
the Hodge decomposition can fail if the variety is not K\"ahler.

\subsection{$p$-adic periods}
Now, let $p$ be a prime and let 
$\C_p$ be the completion of an algebraic closure $\Qbar_p$ of $\Q_p$ for the $p$-adic valuation,
and $\Q_p^{\rm nr}\subset\Qbar_p$ be the maximal unramified extension of $\Q_p$.

Let $K\subset\C_p$ be a finite extension of $\Q_p$, and let $G_K={\rm Aut}_{\rm cont}(\C_p/K)=
{\rm Gal}(\Qbar_p/K)$ be the absolute Galois group of $K$.

 The Hodge theory for varieties over $K$ was initiated by Tate in his 
incredibly influential paper~\cite{Tate}
and given a definitive form by Fontaine (see ~\cite{Fo82,Fo83,Fo94b}) who introduced various rings 
$\bcris$, $\bst$, $\bdr$, in which the periods
of algebraic varieties over $K$ should live. The situation is more complicated here than over $\C$:   The results over $\C$ sketched above show that the periods of
varieties defined over subfields of $\C$ live in $\C$. But 
 in the $p$-adic case, as Tate showed,  these periods do not, in general,
live in $\C_p$ (for example, the $p$-adic avatar of $2\pi i$ does not live in $\C_p$ -- this is a
rephrasing of Tate's result~\cite{Tate} that\footnote{If $r\in\Z$, and $V$ is a $G_K$-module,
we let $V(r)$ be its $r$-th Tate twist: i.e., the module $V$ with action
of $\sigma\in G_K$ multiplied by $\chi^r(\sigma)$, where 
$\chi:G_K\to\Z_p^\dual$ is the cyclotomic character.} $H^0(G_K,\C_p(1))=0$).

If $X$ is a smooth and proper algebraic variety defined over $K$,
there is no direct analog of $H^i_{\rm B}(X(\C),\Q)$; what plays its role in the $p$-adic
world is the \'etale cohomology $H^i_{\eet}(X_{\C_p},\Q_p)$. Indeed, if $X$ is defined over a subfield $F$ of $K$
that embeds into $\C$, and if we choose an embeding into $\C$
of the algebraic closure $\overline F$ of $F$ in $\C_p$, we have natural isomorphisms 
(Artin~\cite{Art} for the last one)
$$H^i_{\eet}(X_{\C_p},\Q_p)\simeq H^i_{\eet}(X_{\overline F},\Q_p)\simeq 
H^i_{\eet}(X_\C,\Q_p)\simeq \Q_p\otimes_\Q H^i_{\rm B}(X(\C),\Q)$$

Fontaine's rings $\bcris\subset\bst\subset\bdr$ are topological rings endowed
with compatible continuous actions of $G_K$, and containing a $p$-adic avatar $t$ of $2\pi i$
(Fontaine's $2\pi i$) on which $\sigma\in G_K$ acts by $\sigma(t)=\chi(\sigma)t$.
Moreover $\bdr^{G_K}=K$ and $\bcris^{G_K}=\bst^{G_K}=K\cap\Q_p^{\rm nr}$,
and:

$\bullet$
$\bdr=\bdr^+[\frac{1}{t}]$, where $\bdr^+$ is complete for the $t$-adic topology
which is stronger that the natural topology (for which $\ovk$ is dense~\cite{dense}),
we have $\bdr^+/t=\C_p$, and $\bdr^+$ is the universal infinitesimal thickening of $\C_p$ (see~\cite{padova2}). 
 The filtration $\bdr^i:=t^i\bdr^+$ is stable by $G_K$.

$\bullet$ $\bst=\bcris[\log\tilde p]$ is endowed with a Frobenius $\varphi$
satisfying $\varphi(t)=pt$, $\varphi(\log\tilde p)=p\,\log\tilde p$,
and a $\bcris$-derivation $N$ with $N(\log\tilde p)=-1$ (hence $N\varphi =p\varphi N$);
$\varphi$ and $N$ commute with $G_K$.

$\bullet$ $\Q_p$ can be recovered inside $\bst$ using the above structure:  we have the
fundamental exact sequence 
$$0\to\Q_p\to \bst^{N=0,\varphi=1}\to\bdr/\bdr^+\to 0$$
Hence, in particular, $\Q_p={\rm Fil}^0(\bst^{N=0,\varphi=1})$. 

\begin{conjecture}\label{ht3}
{\rm (Fontaine~\cite{Fo82,Fo94b})}
We have functorial  period isomorphisms commuting with all structures {\rm ($G_K$, filtration, $\varphi$
and $N$)}
\begin{align*}
{\rm(C_{dR})}\quad &\bdr\otimes_KH^i_{\rm dR}(X)\simeq \bdr\otimes_{\Q_p}H^i_{\eet}(X_{\C_p},\Q_p)\\
{\rm(C_{st})}\quad &\bst\otimes_{\Q_p^{\rm nr}}H^i_{\rm HK}(X)\simeq \bst\otimes_{\Q_p}H^i_{\eet}(X_{\C_p},\Q_p)
\end{align*}
\end{conjecture}
In the conjecture,
$H^i_{\rm HK}(X)$ (the Hyodo-Kato cohomology) is an extra-structure on $H^i_{\rm dR}(X)$
coming from ``reduction modulo $p$'': it is a $\Q_p^{\rm nr}$-module equipped with a semi-linear $\varphi$,
a monodromy operator $N$ with $N\varphi=p\varphi N$, and a smooth action of $G_K$ commuting with $\varphi$
and $N$,
and we have a (Hyodo-Kato)-isomorphism 
$$H^i_{\rm dR}(X)\simeq (\Qbar_p\otimes_{\Q_p^{\rm nr}}H^i_{\rm HK}(X))^{G_K}$$

If $X$ has good reduction, then $N=0$, the inertia group acts trivially, and Hyodo-Kato cohomology
is crystalline cohomology of the reduction mod $p$ 
developed by Grothendieck and Berthelot~\cite{Gcris,B}; 
if $X$ has a semistable model, then inertia group still acts trivially and $H^i_{\rm HK}(X)$
is the log-crystalline cohomology of the reduction modulo~$p$ as defined by Hyodo and Kato~\cite{HK}. 
In general, $H^i_{\rm HK}(X)$ has been defined by Beilinson~\cite{BE2},  
using de Jong's alterations~\cite{dJ}  to reduce to the semistable case.

\begin{remark}\label{ht4}
(i) All the cohomology groups involved in Conjecture \ref{ht3} are finite dimensional over
their respective fields.

(ii) This conjecture is now a theorem (see \cite{Ts}, \cite{Fa94}, \cite{Ni08}, \cite{BE2}). All these
proofs construct the period map, show compatibility with cycle class maps and traces, and use Poincar\'e duality
to conclude that the period  map is an isomorphism.

(iii) By taking invariants under suitable structures, the period  
isomorphisms from Conjecture~\ref{ht3} give recipes to
recover $(H^i_{\rm HK}(X),H^i_{\rm dR}(X))$ from $H^i_{\eet}(X_{\C_p},\Q_p)$ and vice-versa:
one has isomorphisms (of filtered $K$-modules for the first one,
of $(\varphi,N,G_K)$-modules over $\Q_p^{\rm nr}$ for the second one) 
\begin{align*}
(\bdr\otimes_{\Q_p}H^i_{\eet}(X_{\C_p},\Q_p))^{G_K}&\simeq H^i_{\rm dR}(X),\\   
\varinjlim\nolimits_{[L:K]<\infty} (\bst\otimes H^i_{\eet}(X_{\C_p},\Q_p))^{G_L}&\simeq H^i_{\rm HK}(X),
\end{align*}
and an exact sequence
$$0\to H^i_{\eet}(X_{\C_p},\Q_p)\to (\bst\otimes_{\Q_p^{\rm nr}}H^i_{\rm HK}(X))^{N=0,\varphi=1}
\to (\bdr\otimes_KH^i_{\rm dR}(X))/{\rm Fil}^0\to 0$$

(iv) By taking ${\rm Fil}^0$-part of both terms in ${\rm C}_{\rm dR}$, and modding out by $t$,
we obtain the following
result which was the initial formulation of the Hodge-Tate conjecture by Tate~\cite{Tate}
 (renamed ${\rm C}_{\rm HT}$ by Fontaine), and which should be compared\footnote{There is
no twist in (\ref{ht1}) but, actually, if one wants cycle classes to match one needs
to introduce the same powers of $2\pi i$ as for the twists in $({\rm C}_{\rm HT})$: one can write
$\C_p(-r)=t^{-r}\C_p$, and $t$ is the $p$-adic avatar of $2\pi i$.} to~(\ref{ht1}): 
we have a functorial $G_K$-equivariant decomposition:
$${\rm(C_{HT})}\quad 
{\C_p}\otimes_{\Q_p}H^i_{\eet}(X_{\C_p},\Q_p)\simeq\oplus_{r+s=i}{\C_p}({-r})\otimes_K H^s(X,\Omega^r)$$

(v) An interested  reader might want to consult  \cite{Ni21} for a discussion  of the development of Hodge Theory of $p$-adic algebraic varieties in the 50 years following Tate's article. 
\end{remark}

\subsection{$p$-adic analytic varieties}
The theory of $p$-adic analytic varieties was also initiated by Tate~\cite{Rigid}: 
Tate's paper
was finally published in 1971 thanks to Serre's persistence (see~\cite{STcor}, Tate 07/10/1969),
 but the results were circulating as ``Rigid analytic spaces,
private notes of J. Tate, reproduced with(out) his permission by IHES'' starting from 1962 
(see~\cite{STcor},
Tate 16/10/1961, Tate 26/10/1961, Serre 10/11/1961, Tate 14/11/1961, Serre x/04/1962).

   In~\cite{Tate}, Tate already stated his Conjecture ${\rm C}_{\rm HT}$ for 
the $p$-adic \'etale cohomology of proper analytic varieties. 
For quasi-compact varieties -- e.g., affinoids or proper varieties -- \'etale cohomology
coincides with pro-\'etale cohomology.
To compute both  cohomologies one uses a hypercovering by affinoids; affinoids are $K(\pi,1)$-spaces~\cite{Sch}
hence, locally, (pro-)\'etale cohomology reduces to continuous cohomology of the
fundamental group. The difference being that:

--- for $H^i_{\eet}(X_{\C_p},\Q_p)$, one uses $\Z_p$
as coefficients on each affinoid and inverts $p$ at the end, 

--- for $H^i_{\proeet}(X_{\C_p},\Q_p)$, 
one uses $\Q_p$ as coefficients on each affinoid. 

It is easier to work with $\Q_p$-coefficients since rational $p$-adic Hodge Theory 
is much easier and much more robust than the integral one.

\vskip2mm
The pro-\'etale $p$-adic cohomology of proper analytic varieties (which is also the \'etale
cohomology as we mentioned above) behaves in a very similar way to the $p$-adic \'etale cohomology
of algebraic varieties (note, however, that this is not the case if one considers nontrivial
coefficients): 

$\bullet$ Scholze~\cite{Sch0,Sch}
 proved that the $\Q_p$-vector spaces $H^i_{\proeet}(X_{\C_p},\Q_p)$ are finite dimensional.
(This is not the case anymore for nontrivial coefficients as was realized some years
later (see~\cite{hansen}):
the cohomology groups
are Banach-Colmez spaces~\cite{ALBM,LNRZ} (BC's for short, see~\cite{CB,CF}),
hence, in general, only finite Dimensional (see Section~\ref{geo2.1}).)

$\bullet$ Scholze also proved ${\rm C}_{\rm dR}$, 
with the same formulation as in Conjecture~\ref{ht3},
but he found a way to circumvent Poincar\'e duality which was not available at that time.

$\bullet$ The authors  proved ${\rm C}_{\rm st}$ (see~\cite{CN1} in the case $X$ has a semistable model\footnote{See also \cite{BMS} and \cite{CK} for a different approach.},
and~\cite{CN5} for the general case), using crucially the finiteness result of Scholze
plus some fragments of the theory~\cite{CB,CF} of BC's (as well as the basic comparison theorem,
i.e., Theorem~\ref{basic1} below). 

$\bullet$ Poincar\'e duality for rational pro-\'etale cohomology can be deduced from 
that for de Rham cohomology via comparison isomorphisms~\cite{LLZ} or from integral Poincar\'e dualities of 
Mann~\cite{Mann} and Zavyalov~\cite{Zav}.

\vskip2mm
In the rest of the text, we explore what happens for pro-\'etale
cohomology of partially proper analytic varieties (for example, analytifications
of nonproper algebraic varieties).  As the reader will see, it  behaves very differently
(the case of the open unit disk gives a good illustration of the new phenomena that appear,
see Section~\ref{sad1}).
We will mainly be concerned with the geometric situation, but see Sections~\ref{arit1} 
and~\ref{arit2} for results in the arithmetic case.

A big motivation for studying $p$-adic \'etale and pro-\'etale cohomologies of nonproper $p$-adic analytic varieties comes from the potential applications
to the geometrization of the hoped for $p$-adic local Langlands correspondence~\cite{CDN1,CDN3,CDN5},
but we will not elaborate on this in this survey, rather referring the reader to the
survey~\cite{icbs} devoted to this topic.  
One reason for that  is that the applications to the $p$-adic local Langlands correspondence use
mainly the $p$-adic \'etale cohomology rather than the pro-\'etale one, and the results
that we have about $p$-adic \'etale cohomology so far are much more fragmentary than
for the pro-\'etale one.

 \subsubsection*{Notation and conventions.}\label{Notation}
 Let $p$ be a prime and let $K$ be a complete discrete valuation field with a perfect residue field, of mixed characteristic. 
 Let $\so_K$ be the ring of integers in~$K$, and $k$ be its
residue field. 
Let $W(k)$ be the ring of Witt vectors of $k$ and let $F$ be its
fraction field (i.e., $W(k)=\so_F$).   

Let $\ovk$ be an algebraic closure of $K$ and let $\so_{\ovk}$ denote the integral closure of $\so_K$ in $\ovk$. Let $C=\wh{\ovk}$ be the completion of $\ovk$ for the $p$-adic valuation,
and let $G_K:={\rm Aut}_{\rm cont}(C/K)=  
\Gal(\overline {K}/K)$.

Let $\breve{C}={\rm Frac}(W(\overline{k}))\subset C$ 
and $C^{\rm nr}=\cup_{[k':k]<\infty}W(k')[\frac{1}{p}]\subset\ovk$
(hence $\breve{C}$ is the completion of~$C^{\rm nr}$),
 and let $\phi$ be the absolute
Frobenius on $\breve C$ and $C^{\rm nr}$.

    All rigid analytic spaces and dagger spaces considered will be over $K$ or $C$;  we assume that they are separated, taut, and countable at infinity. 
    
     We will use condensed mathematics as developed in  \cite{Sch19}, \cite{Sch20}; we will write
    $\sd(\Q_{p,\Box})$ for the $\infty$-derived category of solid $\Q_p$-vector spaces, etc.

 
 We will use the bracket notation for certain limits:
  $[C_1\stackrel{f}{\to} C_2]$ denotes the mapping fiber of $f$. 
  
\section{The basic comparison theorem}
We start our survey with a discussion of a basic observation: $p$-adic pro-\'etale cohomology of analytic varieties can be computed via their refined de Rham cohomology. 
\subsection{An example} We will start with a motivating example. 
\subsubsection{The complex unit disk}
Let $Y$ be the open unit disk over $\C$ (field of complex numbers) viewed as a complex
analytic variety of dimension~$1$.  The de Rham and singular (Betti) cohomology groups have the familiar form:
\begin{align*}
&H^i_{\rm dR}(Y)\simeq\begin{cases} \C &{\text{if $i=0,$}}\\ 0 & {\text{if $i\geq 1;$}}\end{cases}
&&H^i_{\rm B}(Y,\Q)\simeq\begin{cases} \Q &{\text{if $i=0,$}}\\ 0 & {\text{if $i\geq 1;$}}\end{cases}\\
&H^i_{{\rm dR},c}(Y)\simeq\begin{cases} \C &{\text{if $i=2,$}}\\ 0 & {\text{if $i\neq 2;$}}\end{cases}
&&H^i_{{\rm B},c}(Y,\Q)\simeq\begin{cases} \Q &{\text{if $i=2,$}}\\ 0 & {\text{if $i\neq 2.$}}\end{cases}
\end{align*}
Moreover, we have natural isomorphisms 
$$H^i_{\rm dR}(Y)\simeq H^i_{\rm B}(Y,\Q)\otimes_\Q\C ,
\quad
H^i_{{\rm dR},c}(Y)\simeq H^i_{{\rm B},c}(Y,\Q)\otimes_\Q\C,\quad{\text{for all $i\geq 0$}},$$
and the natural pairings
\begin{align*}
H^i_{\rm dR}(Y)\times H^{2-i}_{{\rm dR},c}(Y)&\to H^2_{{\rm dR},c}(Y)\simeq\C,\\
H^i_{\rm B}(Y,\Q)\times H^{2-i}_{{\rm B},c}(Y,\Q)&\to H^2_{{\rm B},c}(Y,\Q)\simeq\Q
\end{align*}
are perfect  (i.e., we have a Poincar\'e duality).

\subsubsection{The $p$-adic unit disk}\label{sad1}
Let now $Y_{\C_p}$ be the open unit disk over $\C_p$ (completion of the algebraic closure of $\Q_p$)
viewed as a rigid
analytic variety of dimension~$1$.  Then we have the same results as above for de Rham cohomology
(with $\C$ replaced by $\C_p$).  The role of Betti cohomology in this setting  is played by
the $p$-adic (pro-)\'etale cohomology $H^*_{\proeet}(-,\Q_p)$ with $\Q_p$-coefficients. 
By \cite[Th. 6.14]{CN5}, \cite[Cor.\,1.6]{AGN},   we have
$$
H^i_{\proeet}(Y_{\C_p},\Q_p)\simeq 
\begin{cases} \Q_p &{\text{if $i=0$,}}\\ \O(Y_{\C_p})/\C_p &{\text{if $i=1$,}}\\
0 &{\text{if $i\geq 2$.}}\end{cases}
$$
$$
H^i_{{\proeet},c}(Y_{\C_p},\Q_p)\simeq 
\begin{cases} 0 &{\text{if $i\neq 2$,}}\\
[\O(\partial Y_{\C_p})/\O(Y_{\C_p})\hskip.5mm\rule[1mm]{3mm}{.6pt}\hskip.5mm\Q_p] &{\text{if $i=2$.}}
\end{cases}
$$
Here the bracket $[(-)-(-)]$ denotes an extension. The notation 
 $\partial Y_{\C_p}$ refers to the rigid analytic ghost circle, boundary of the open unit disk;
it is obtained by removing from $Y_{\C_p}$ all the affinoids $W$ that it contains:
we have $\O(\partial Y_{\C_p})=\varinjlim_W\O(Y_{\C_p}\moins W)$ and $\partial Y_{\C_p}$ behaves
like a proper variety of ``real'' dimension~$1$ (see Theorem \,\ref{main-arithmeticY0}).

\vskip2mm
  The main differences with the case of the complex unit disk are the following:

$\bullet$ $H^1_{\proeet}(Y_{\C_p},\Q_p)$ and $H^2_{{\proeet},c}(Y_{\C_p},\Q_p)$ are huge
groups and are not invariant by base change to a bigger complete algebraically closed extension $C$
of $\C_p$ (and there are many of those: there is no $p$-adic analog of the celebrated Gelfand-Mazur theorem), as the
``$\C_p$-components'' $\O(Y_{\C_p})/\C_p$ and $\O(\partial Y_{\C_p})/\O(Y_{\C_p})$ are
base changed to $C$.

$\bullet$ The groups $\O(Y_{\C_p})/\C_p\simeq H^0(Y_{\C_p},\Omega^1)$ (via $f\mapsto df$)
 and $\O(\partial Y_{\C_p})/\O(Y_{\C_p})\simeq H^1_c(Y_{\C_p},\O)$ are in
Serre duality over $\C_p$, but this cannot be turned into a duality over $\Q_p$
since $[\C_p:\Q_p]=\infty$. Hence the existence of a Poincar\'e duality
$$H^i_{\proeet}(Y_{\C_p},\Q_p)\times H^{2-i}_{{\proeet},c}(Y_{\C_p},\Q_p)\to\Q_p$$
looks impossible
(not to  mention the fact that the numerology does not cooperate either).

\begin{remark} Note that, by contrast, for $\ell\neq p$, the results are as expected:
$$
H^i_{\proeet}(Y_{\C_p},\Q_\ell)\simeq
\begin{cases} \Q_\ell &{\text{if $i=0$,}}\\ 0 &{\text{if $i\geq 1$;}}
\end{cases}
\quad
H^i_{{\proeet},c}(Y_{\C_p},\Q_\ell)\simeq
\begin{cases} 0 &{\text{if $i\neq 2$,}}\\
\Q_\ell  &{\text{if $i=2$}}
\end{cases}
$$
and we have a Poincar\'e duality
$$H^i_{\proeet}(Y_{\C_p},\Q_\ell)\times H^{2-i}_{{\proeet},c}(Y_{\C_p},\Q_\ell)\to\Q_\ell.$$
\end{remark}

 We will explain how to remedy the two difficulties mentioned above.  
For the fist one, we show (see Section~\ref{geo1}, in particular note~\ref{idea}) that
$H^i_{{\proeet}}(Y_{\C_p},\Q_p)$ is actually the $\C_p$-points of a functor built out of
copies of $\Q_p$ and ${\bf G}_a$ (it is this ${\bf G}_a$ part that varies with $\C_p$).
For the second one,  the clue is to consider not ${\rm Hom}(-,\Q_p)$ as what one would naturally do,
but derived Hom, ${\rm RHom}_{\rm TVS}(-,\Q_p)$ in the category ${\rm TVS}$ (see Section~\ref{geo2})
in which the functor
attached to $H^i_{{\proeet}}(Y_{\C_p},\Q_p)$ naturally lives.

\subsection{Slogans} 
We turn now to  key slogans  in  Hodge Theory of $p$-adic analytic varieties.
In this section $Y$ is  a smooth rigid analytic variety over $K$. 
\subsubsection{Comparison theorems as  refined Lazard isomorphisms} \label{slonce1} (Formally) similarly to the algebraic case, pro-\'etale cohomology  can be computed by data coming from refined de Rham cohomology
\begin{equation}\label{kol1}
\R\Gamma_{\proeet}(Y_C,\Q_p) \longleftrightarrow \{\R\Gamma_{\dr}(Y) + {\rm Fil}^{\jcdot},\phi,N, \rho\}
\end{equation}
This data consist of the de Rham cohomology itself together with its Hodge filtration $\rm Fil^{\jcdot}$ but also the refined structures of Frobenius $\phi$, monodromy $N$, and a residual Galois action $\rho$ that satisfy compatibility relations.
 The structures $(\phi,N,\rho)$ live on Hydo-Kato cohomology $\rg_{\hk}(Y_C)$,
 which is a $C^{\nr}$-avatar of de Rham cohomology whose existence is due to the possibility
of reducing modulo~$p$ (models of) our analytic varieties, and whose cohomology groups,
 under favorable conditions (for example when $Y$ is partially proper),
 are subgroups of de Rham cohomology groups. 

\vskip1mm
   The (by now standard, see \cite{CN1}, \cite{SG}, \cite{CN4}, \cite{BMS}) passage between the two sides of \eqref{kol1} can be briefly described as follows:

$\bullet$ Locally on $Y$, by the $K(\pi,1)$-lemma, we can think of $\R\Gamma_{\proeet}(Y_C,\Q_p)$ as the continuous group cohomology $\R\Gamma(\pi_1(Y_C),\Q_p)$.
 
$\bullet$ Then, by the mixed characteristic Artin-Schreier, we can pass to group cohomology 
of one of the (perfectoid) period rings.

$\bullet$  Now, using almost purity, we can replace $\pi_1(Y_C)$ by a commutative\footnote{In the
arithmetic situation~\cite{CN1}, this group is no longer abelian, but is the semi-direct product
of an open subgroup of $\Z_p^\dual$ by $\Z_p^d$.} 
 $p$-adic Lie group $\Gamma$, which is the Galois group of the extension obtained by
extracting $p^\infty$-roots of coordinates\footnote{This choice of coordinates causes some
difficulties when one wants to globalize these local computations: 
on intersections (i.e., fiber products)
of affinoid  coverings one gets too many coordinates; the extra coordinates
disappear thanks to the perfectoid magic.}
 of $Y$ (hence $\Gamma\simeq\Z_p^d$, where $d=\dim Y$),
and the period ring by the (perfectoid) period ring attached to this extension.

$\bullet$ Then, by decompletion, we can deperfect the period ring to bring 
it closer to the original variety $Y$, and make the action of $\Gamma$ locally analytic.
 Invoking theorems of Lazard~\cite{laz}
 we can now pass to Lie algebra cohomology
$\R\Gamma({\rm Lie}\,\Gamma,-)$, 
which is computed via a Koszul complex yielding the de Rham complex, but where the differential
$d$ is multiplied 
\footnote{It is  this multiplication by $t$ that is responsible 
for the huge $\C_p$-vector spaces that appear in $H^1_{\proeet}$ 
for the  unit disk in Section~\ref{sad1}.}
by $t$.

\vskip1mm
  The above steps can be reversed after stabilization, i.e.,  Tate-twisting the pro-\'etale cohomology by a high enough twist and reflecting this on the other side by twisting all the structures.

  \begin{remark}  (1) The above  passage  uses most  of the standard tools from  the toolbox of modern $p$-adic Hodge theorists.
 It is surprising  that most  of these tools were developed to study Hodge Theory of $p$-adic algebraic varieties but they work very well in the analytic setting as well.
 The authors were in particular surprised by the fact that the syntomic techniques developed by Fontaine-Messing, Hyodo, Kato, Tsuji, and others to compute $p$-adic nearby cycles turned out to be  robust enough to treat analytic varieties. 
  
  (2) There is one major exception here: the almost purity theorem that comes from the realm of "relative almost \'etale theory" and which Faltings was not able to prove.
 This theorem was proved by Kedlaya--Liu and Scholze only recently\footnote{During the Hot Topics workshop at MSRI on homological conjectures (in 2018) Faltings mentioned in his talk that he tried to prove the almost purity theorem by passing from characteristic zero to characteristic $p$; it did not occur to him that one could try to go the other way (the way it was proved).} and it unlocked the Hodge Theory of $p$-adic analytic varieties: in the absence of resolution of singularities  in mixed characteristic it is the second best thing.
\end{remark}

\subsubsection{Hyodo-Kato cohomology as an avatar of $\ell$-adic cohomology} The Hyodo-Kato cohomology $\R\Gamma_{\hk}(Y_C)$ mentioned above should be thought of as an avatar of $\ell$-adic \'etale cohomology.
 In the case it is of finite rank over $C^{\nr}$ this can be made precise: 
the Hyodo-Kato cohomology groups 
yield  Weil-Deligne representations  (see \cite{Font})  
and it is conjectured~\cite[\S\,2.4.3]{Font} that these representations 
are isomorphic to the ones arising from $\ell$-adic cohomology $\R\Gamma_{\eet}(Y_C,\Q_{\ell})$, for $\ell\neq p$.
 This conjecture is known in some cases (notably in dimension $1$: 
see \cite{Och} in the algebraic setting and~\cite[Th.\,8.2]{CDN2}; but also for projective varieties with good reduction over absolutely unramified base by Katz-Messing \cite{KM})  
and allows to bootstrap results from $\ell$-adic cohomology to $p$-adic cohomology (as in \cite{CDN1}).
\subsubsection{\'Etale versus pro-\'etale cohomology} There is a natural map
$$
\epsilon:\quad \R\Gamma_{\eet}(Y_C,\Q_p)\to \R\Gamma_{\proeet}(Y_C,\Q_p)
$$
from \'etale to pro-\'etale cohomology
but it is quite mysterious.
 It is not always injective as one could naively expect (see Remark \ref{hk4} for an example).
 Moreover, the \'etale cohomology itself is difficult to understand.
 This is because, unlike pro-\'etale cohomology which, via comparison theorems,
  uses mostly rational $p$-adic Hodge Theory,
 the \'etale one needs integral (or integral up to universal denominators) $p$-adic Hodge Theory,
 which is much more subtle.

  However, in some favorable cases the \'etale cohomology 
can be recovered from the pro-\'etale cohomology: 
see \cite{CDN3} for the Drinfeld space and \cite{CDN1} 
for coverings of the Drinfeld space in dimension~$1$. In the latter have we have $H^1_{\eet}\simeq H_{\proeet}^{1, G-\mathrm{bdd}}$, for $G=\mathrm{GL}_2(K)$,  but we not know whether the same is true in degree $2$, where $H^2_{\proeet}=0$, i.e., we do not know whether $H^2_{\eet}=0$. 

\subsection{The basic comparison theorem}
We will discuss now in more detail the first slogan (see Section \ref{slonce1}).
\subsubsection{Statement of the theorem}
Let $Y_K$ be a smooth rigid analytic variety with an overconvergent (or dagger)
structure.
 Examples of such $Y_K$ are  proper or Stein varieties, the analytifications of  algebraic
varieties,  overconvergent affinoids, etc.  We assume that $Y_K$ is geometrically connected.

  The following theorem states the basic form of the comparison theorem between $p$-adic pro-\'etale cohomology and de Rham cohomology.
\begin{theorem} \label{basic1}
{\rm (Basic comparison theorem \cite{CN4})}
Let $n\leq r$ be positive integers.
 We have a long exact sequence in $\sd(\Q_{p,\Box})$
\begin{equation}\label{basic2}
\cdots\to X^{r,n-1}\to {\rm DR}^{r,n-1}\to H^n_{\proeet}(Y_C,\Q_p(r))\to
X^{r,n}\to {\rm DR}^{r,n}\to\cdots
\end{equation}
where
$$X^{r,i}:=(H^i_{\rm HK}(Y_C)\otimes^{\Box}_{\Q_p} \bst^+)^{N=0,\varphi=p^r},
\quad
{\rm DR}^{r,i}:=H^i(\rg_{\rm dR}(Y_K)\otimes^{\Box}_{K}\bdr^+/{\rm Fil}^r).
$$
\end{theorem} 
Here, the Hyodo-Kato cohomology $H^i_{\rm HK}(Y_C)$, just as in the algebraic setting,  is an additional structure on de Rham cohomology that comes
from the ``reduction modulo $p$''.  More precisely,
$H^i_{\rm HK}(Y_C)$ is a $C^{\rm nr}$-module, endowed with actions of $G_K$, $\varphi$ and $N$,
such that the action of $G_K$ commutes with $\varphi$ and $N$, and $N\varphi=p\varphi N$.  Moreover   we have
a ``Hyodo-Kato isomorphism'' $$\iota_{\hk}:H^i_{\rm HK}(Y_C)\otimes^{\Box}_{C^{\rm nr}}C\stackrel{\sim}{\to} H^i_{\rm dR}(Y_C).$$
The  construction of this Hyodo-Kato cohomology  is modeled on the one of Beilinson for algebraic varieties (see \cite{BE2}): it is the classical  Hyodo-Kato cohomology  if $Y_K$
has a semistable model and uses the \'etale local alterations of Hartl and Temkin (see \cite{Urs}, \cite{Tem}) to reduce to that case. 
Topologically, $H^i_{\rm HK}(Y_C)$ is a projective limit of finite dimensional
$C^{\rm nr}$-modules on which the action of $G_K$ is smooth (so the action on
$H^i_{\rm HK}(Y_C)$ itself is ``pro-smooth''). This allows to control the topology on many other $p$-cohomologies.

\begin{remark}{\rm (Other definitions of Hyodo-Kato cohomology)} The Beilinson-style Hyodo-Kato cohomology described above was defined in \cite{CN4}.  Since then two alternate constructions have appeared: 
\begin{enumerate}[leftmargin=*] 
 \item  The first one is  due to  Binda-Gallauer-Vezzani (see \cite{BGV}, \cite{BKV}, and \cite{BV} for a survey).     It lives in the world  of rigid analytic motives
of Ayoub~\cite{ayoub}.  Its  definition also uses the alterations of Hartl and Temkin to reduce to "simpler" rigid analytic varieties. But it has two major advantages over the definition \`a la Beilinson: 

(i) Via the theory of rigid analytic motives it blackboxes the reduction steps via alterations.

(ii) It allows to define 
 Hyodo-Kato cohomology for, for example, perfectoid spaces.

\item Very recently, Ansch\"utz-Bosco-Le Bras-Rodriguez Camargo-Scholze \cite{ABLBRCS} constructed Hyodo-Kato cohomology via the theory of Gelfand  stacks. 
This definition also allows to work in a more general setting than rigid analytic varieties. 
 \end{enumerate}
\end{remark}

 Theorem \ref{basic1} expresses   $p$-adic pro-\'etale cohomology in terms of differential forms but in a derived fashion,
 i.e., it allows to recover the derived $p$-adic pro-\'etale cohomology from the refined derived de Rham cohomology. An important fact to notice here is that the period rings appear in their $+$-form,
 i.e., we do not invert $t$. This accounts for  the $n\leq r$ condition. 
A condition that seems restrictive but it will allow us to recover the pro-\'etale cohomology groups ({\bf sic !})  from the refined de Rham cohomology groups 
(conjecturally in general, but inconditionaly in many cases of interest;
see the examples in Section \ref{nowy1}  and the Conjecture ${\rm C}_{\st}$ in Section \ref{nowy2}).

\begin{remark}{\rm (Analytic varieties over $C$)}\label{ht8}
If we do not assume that $Y_C$ comes from a variety over~$K$, we have a similar result
but in the definition of ${\rm DR}^{r,n}$, one has to replace $\rg_{\rm dR}(Y_K)\otimes^{\LL_{\Box}}_K\bdr^+$
with the $\bdr^+$-cohomology.
\end{remark}

\subsubsection{Examples} \label{nowy1} We have the following consequences of Theorem \ref{basic1}:

 (i) If $Y_K$ is proper, the sequence \eqref{basic2} splits into short exact sequences
$$0\to  H^n_{\proeet}(Y_C,\Q_p(r))\to
X^{r,n}\to {\rm DR}^{r,n}\to 0,$$
which is reminiscent of the fundamental exact sequence
\begin{equation}\label{fund1}
0\to \Q_p(r)\to \bst^{+,N=0,\varphi=p^r}\to\bdr^+/{\rm Fil}^r\to 0.
\end{equation}

 Moreover $H^n_{\proeet}(Y_C,\Q_p(r))$ is finite dimensional over $\Q_p$ and
$H^n_{\rm HK}(Y_C)$ is finite dimensional over $C^{\rm nr}$.
Hence this case is completely analogous to what happens
for $p$-adic \'etale cohomology of algebraic varieties.
The finite dimensionality of $H^n_{\proeet}(Y_C,\Q_p(r))$ was proved by Scholze in \cite{Sch};
the splitting of the exact sequence uses this finite dimensionality, 
the fact that  the maps in the exact sequence \eqref{basic2} are
induced by  morphisms of Banach-Colmez spaces, plus basic properties of BC's.

(ii) If $Y_K$ is the affine space ${\mathbb A}^d_K$ of dimension $d$, we have
$H^n_{\rm HK}(Y_C)=0$,  for $n\geq 1$, from which we deduce an isomorphism
$$H^n_{\proeet}(Y_C,\Q_p(r))\simeq (\Omega^{n-1}(Y_C)/{\rm Ker}\,d)(r-n),
\quad \text{for } 1\leq n\leq r.$$
By contrast, one can show that $H^n_{\eet}(Y_C,\Q_p(r))=0$, for  $n\geq 1$.

(iii) More generally, if $Y_K$ is Stein (i.e., a strictly increasing union of
affinoids, which implies that  $H^i(Y_C,{\cal F})=0$, for $i\geq 1$ and ${\cal F}$ a coherent
sheaf on $Y_C$), we have a short exact sequence
$$0\to \big(\Omega^{n-1}(Y_C)/{\rm Ker}\,d\big)(r-n)\to H^n_{\proeet}(Y_C,\Q_p(r))\to X^{r,n}\to 0$$

\begin{remark}{\rm (Not spherically complete puzzle)}\label{hk4}
It follows from example (iii) above that, if $Y$ is the open unit ball in dimension~$1$, then
$H^2_{\proeet}(Y_C,\Q_p)=0$.  On the other hand $H^2_{\eet}(Y_C,\Q_p)=0$ if and only if
$C$ is spherically complete~\cite[Th.\,A.6]{CDN2}.
This shows that \'etale cohomology does not always inject
into pro-\'etale cohomology.

   The above dichotomy between spherically complete and not spherically complete fields is a bit
unsettling. But, at least in dimension $1$, it
 can be fixed by using the {\it adoc generic fiber}\footnote{We have 
$\O(Y_C)=\varprojlim_{n\to \infty}\O_C\langle \frac{T}{p^{1/n}}\rangle[\frac{1}{p}]$.
If $Y_C^{\rm ado}$ denotes the generic adoc fiber, then
$\O(Y^{\rm ado}_C)=\varprojlim_{n\to \infty}\big(\O_C\langle \frac{T}{p^{1/n}}\rangle\big)[\frac{1}{p}]=
\O_C[[T]][\frac{1}{p}]$: i.e., $\O(Y^{\rm ado}_C)$ is the ring of {\it bounded} analytic functions
on $Y_C$. If $\partial Y_C$ denotes the (adoc) ghost circle at the boundary, then
$\O(\partial Y_C)=\big(\O_C[[T]][T^{-1}]\big)^{p{\text{-}}\wedge}[\frac{1}{p}]$, and the obvious map
$\O(Y^{\rm ado}_C)\to \O(\partial Y_C)$ does not extend to $\O(Y_C)\to \O(\partial Y_C)$.}
 instead of the rigid one (see~\cite[Rem.\,A.9]{CDN2}):
if we view $Y$ as a rigid analytic space, we need to cover
$Y$ by an increasing union of closed balls, and we miss the ghost circle
at the boundary. 

Another possibility would be to change slightly the definition of \'etale cohomology
to impose that $Y$ is a $K(\pi,1)$-space (more generally, that affine adoc spaces are $K(\pi,1)$):
indeed one can show that\footnote{Here $\pi_1(Y_C)$ is the pro-\'etale fundamental group, i.e., the
profinite \'etale fundamental group.}
 $H^i(\pi_1(Y_C),\Q_p)$
coincides with $H^i_{\eet}(Y_C,\Q_p)$ for $i\leq 1$, and is $0$
if $i\geq 2$ (without imposing $C$ to be spherically complete, see~\cite[Prop.\,A.8]{CDN2}).
\end{remark}
\subsubsection{Comparison with syntomic cohomology}
The long exact sequence \eqref{basic2}   is obtained  from the natural period quasi-isomorphism in $\sd(\Q_{p,\Box})$
\begin{equation}\label{BK1}
\alpha_{\rm BK}:\quad \tau_{\leq r}\rg_{\proeet}(Y_C,\Q_p(r))\simeq \tau_{\leq r}\rg^{\rm BK}_{\synt}(Y_C,\Q_p(r)),
\end{equation}
where 
$$
\rg^{\rm BK}_{\synt}(Y_C,\Q_p(r)):=\big[[\rg_{\rm HK}(Y_C)\otimes^{\LL_{\Box}}_{C^{\rm nr}}\bst^+]^{N=0,\varphi=p^r}
\stackrel{\iota_{\hk}}{\longrightarrow} (\rg_{\rm dR}(Y_K)\otimes^{\LL_{\Box}}_K\bdr^+)/{\rm Fil}^r\big]
$$
is the Bloch-Kato syntomic cohomology. 
The term $(\rg_{\rm dR}(Y_K)\otimes^{\LL_{\Box}}_K\bdr^+)/{\rm Fil}^r$  gives rise to the ${\rm DR}^{r,i}$ term in the  exact sequence \eqref{basic2})
and the term $[\rg_{\rm HK}(Y_C)\otimes^{\LL_{\Box}}_{C^{\rm nr}}\bst^+]^{N=0,\varphi=p^r}$ gives rise to 
 the $X^{r,i}$ terms in the same sequence.

The quasi-isomorphism  \eqref{BK1} follows from  the following two results.
\begin{theorem} {\rm (Pro-\'etale vs syntomic comparison)} \label{kol11}Let $r\geq 0$. The Fontaine-Messing 
period  morphism:
$$
\alpha_{\rm FM}:\quad  \rg^{\rm FM}_{\rm syn}(Y_C,\Q_p(r))\to \rg_{\proeet}(Y_C,\Q_p(r))
$$
is a quasi-isomorphism after truncation $\tau_{\leq r}$.
\end{theorem}
Here, the syntomic cohomology (ala Fontaine-Messing)  is defined as the mapping fiber
\begin{equation}\label{first}
\rg^{\rm FM}_{\rm syn}(Y_C,\Q_p(r)):=\big[[\rg_{\rm cris}({\cal Y})]^{\varphi=p^r}\stackrel{\can}{\longrightarrow}\rg_{\rm cris}({\cal Y})/{\rm Fil}^r
\big]
\end{equation}
where ${\cal Y}$ is an \'etale hypercovering of $Y_C$ by affinoids with
a semistable model\footnote{Sensu stricto, to make this definition independent of choices  one has to take  a colimit over all hypercoverings of such form.} (such hypercoverings can be constructed using the alterations
of Hartl and Temkin) and $\rg_{\rm cris}({\cal Y})$ is the absolute crystalline cohomology of the associated log-formal-schemes.
 The period map $\alpha_{\rm FM}$ is the one originally defined by Fontaine-Messing in \cite{FM}. It involves a version of Poincar\'e Lemma for the appearing period sheaves (a priori in the syntomic-\'etale  topology). 
 
 The proof of Theorem  \ref{kol11} in \cite{CN4} relies on local computations from  \cite{CN1}, \cite{SG}, which use $(\phi,\Gamma)$-modules to express the map $\alpha_{\rm FM}$ as a series of quasi-isomorphisms. 
These modifications are
 then globalized using the approach of all-coordinates borrowed from Bhatt-Morrow-Scholze \cite{BMS}. 
\begin{remark} (Two types of syntomic cohomology) The two types of syntomic cohomology mentioned above (they are quasi-isomorphic) play a different role in $p$-adic Hodge Theory.
 The Bloch-Kato syntomic cohomology 
(called that way because they generalize
 Bloch-Kato Selmer groups from \cite{BK}) is useful for computations: one first tries to get a handle on de Rham cohomology, from that, via the Hyodo-Kato isomorphism on Hyodo-Kato cohomology; then all that remains is to understand the filtration, Frobenius, and monodromy (difficult in general).
 The Fontaine-Messing syntomic cohomology, on the other hand, was traditionally used to prove comparison theorems (see \cite{Ts}).
\end{remark}
 Roughly speaking, the Fontaine-Messing  period morphism $\alpha_{\rm FM}$ in Theorem \ref{kol11} is induced by the fundamental exact sequence 
$$
0\to  \Q_p(r)\to F^r{\mathbb B}^+_{\crr}\stackrel{p^r-\phi}{\longrightarrow }{\mathbb B}^+_{\crr}\to 0
$$
of sheaves on the syntomic site.  
To prove that it is a quasi-isomorphism in a stable range one replaces this sequence  with the exact sequence
$$
0\to \Q_p(r)\to {\mathbb B}_{I}(r)\stackrel{1-\phi}{\longrightarrow} {\mathbb B}_{I^{\prime}}(r)\to 0
$$
of pro-\'etale sheaves and proceeds as sketched in Section \ref{slonce1}.
  Here $I,I^{\prime}\subset (0,\infty)$ are compact intervals with rational endpoints   (conveniently chosen) and ${\mathbb B}_{I}, {\mathbb B}_{I^{\prime}}$ are relative versions of the period rings $\B_{I}, \B_{I^{\prime}}$ ---  coordinate rings of standard affinoids on the $Y_{\rm FF}$-curve (the $Y$-curve of Fargues-Fontaine).
In particular, the choices of
the intervals $I, I^{\prime}$ imply that $t$ has one zero on ${\rm Spa}(\B_{I} )$ and is a unit in $\B_{I^{\prime}}$.

\begin{theorem}\label{ht9}
\begin{enumerate}[leftmargin=*]
\item {\rm (K\"unneth formula for derived de Rham cohomology)} We have a natural quasi-isomorphism:
$$
\rg_{\rm cris}({\cal Y})/{\rm Fil}^r\simeq (\rg_{\rm dR}(Y_K)\otimes^{\LL_{\Box}}_K\bdr^+)/{\rm Fil}^r
$$
\item {\rm (Hyodo-Kato rigidity)} We have  a natural Frobenius equivariant quasi-isomorphism:
\begin{equation}\label{second}\rg_{\rm cris}({\cal Y})
\simeq[\rg_{\rm HK}(Y_C)\otimes^{\LL_{\Box}}_{C^{\rm nr}}\bst^+]^{N=0}
\end{equation}
\end{enumerate}
\end{theorem}
The first claim follows from the reinterpretation of (Hodge completed) crystalline cohomology as (Hodge completed) derived de Rham cohomology, from the K\"unneth formula for the latter, and from the fact that $\B^+_{\dr}/{\rm Fil}^r\simeq \rg_{\dr}(\so_{\C_p})/{\rm Fil}^r$.
 The second claim is argueably the most difficult fact in the theory of Hyodo-Kato cohomology.
 It relies on the existence of a section to the canonical projection $\rg_{\rm cris}({\cal Y})\to \rg_{\rm HK}(Y_C)$; a section that should be compatible both with Frobenius and monodromy and also highly functorial.
 Morally speaking, this follows from a version of Dwork Lemma using the fact that Frobenius is an automorphism on the Hyodo-Kato cohomology but is highly nilpotent on the kernel of the projection.
 In the case Hyodo-Kato cohomology is of finite rank Beilinson in \cite{BE2} has written 
down an elegant abstract argument
why this works.
 In general, one has to upgrade the original construction of Hyodo-Kato, which is highly nonconstructive.

\subsubsection{Pro-\'etale cohomology of period sheaves} The key computation in the proof of Theorem \ref{kol11} is that of pro-\'etale cohomology of the period sheaves ${\mathbb B}_I$. 
\begin{theorem}  {\rm (Cohomology of ${\mathbb B}_I$)} \label{CGN21}Let $I\subset (0,\infty)$ be a compact interval with rational endpoints. 
\begin{enumerate}[leftmargin=*]
\item There is a natural\footnote{Frobenius  sends $I$ to $I/p$, hence
the quotation marks.}
 ``Frobenius equivariant'' quasi-isomorphism in $\sd(\B_{I,\Box})$
\begin{equation}\label{CGN2}
\tau_{\leq r}\rg_{\proeet}(Y_C,{\mathbb B}_I)  \simeq \tau_{\leq r}[[\rg_{\hk}(Y_C)\{r\}\otimes^{\LL_{\Box}}_{C^{\nr}}\B_{I,\log}]^{N=0}\to  \oplus_{Z(I)}(\rg_{\dr}(Y_K)\otimes^{\LL_{\Box}}_{K}\bdr^+)/{\rm Fil}^r](-r),
\end{equation}
where the index set $Z(I)$ is over the zeros of $t$ in ${\rm Spa}(\B_I)$. 
\item If  $t$ is a unit in $\B_I$ then  there is a natural ``Frobenius equivariant'' 
 quasi-isomorphism in $\sd(\B_{I,\Box})$
$$\rg_{\proeet}(Y_C,{\mathbb B}_I)  \simeq [\rg_{\hk}(Y_C)\otimes^{\LL_{\Box}}_{C^{\nr}}\B_{I,\log}]^{N=0}.
$$
\end{enumerate}
\end{theorem}
\begin{remark}\label{BoscHK}
 Bosco in  \cite[Th. 4.1, Rem. 6.13]{GB2} proved a version of the comparison quasi-isomorphism \eqref{CGN2} for the sheaf ${\mathbb B}:=\varprojlim_I{\mathbb B}_I$, where the $t$-torsion
on the right-hand side of \eqref{CGN2} is incorporated to the left-hand side via the $ \LL\eta_t$-operator. 
\end{remark}

 Let $I\subset (0,\infty)$ be a compact interval with rational endpoints such that $t$ has one zero on ${\rm Spa}(\B_I)$.
 Tensoring both sides of \eqref{CGN2} for $r\geq 2d$ ($d=\dim Y_K$) with $\B_I/t^n\simeq \B^+_{\dr}/t^n$, $n\in\N$,   passing to  limit over $n$, and dropping the twist $r$ since we do not care about Frobenius anymore, one obtains the following:
\begin{corollary}  {\rm (Cohomology of  ${\mathbb B}^+_{\dr}$)}\label{finish1}
\begin{equation}\label{hungry1}
\rg_{\proeet}(Y_C,{\mathbb B}^+_{\dr})  \simeq F^0(\rg_{\dr}(Y_K)\otimes^{\LL_{\Box}}_{K}\bdr).
\end{equation}
\end{corollary}
 This result (in a more general setting) was  derived  in a simpler way by Bosco  in \cite[Th. 1.8]{GB1}  directly from the  ${\mathbb B}_{\dr}$-Poincar\'e Lemma.
 
\subsubsection{Pro-\'etale cohomology with compact support}
It came as a surprise to us  -- though, in hindsight, we should have anticipated it -- that compactly supported $p$-adic pro-\'etale cohomology does not behave vis a vis comparison theorems as well as the usual $p$-adic pro-\'etale cohomology.
 So it was a stroke of luck that in \cite{CDN1}, for example, we started to look for a geometric realization of the $p$-adic local Langlands correspondence in the usual cohomology (contrary to the traditional approach to geometrization of the $\ell$-adic local Langlands correspondence, where compactly supported cohomology is used).  

However not all is lost: we still have the basic comparison theorem and a vestige of the fundamental diagram.
  See Corollary  \ref{affinoids-comp} below for the latter; we will discuss the former here.

 It was not clear initially what the right definition  of  compactly supported $p$-adic pro-\'etale cohomology of a smooth partially proper variety $Y$  should be.
 There were two immediate  guesses (see \cite[Sec. A.2]{CDHN} for a discussion): 
 \begin{align}\label{main-def1}
 \R\Gamma^{(1)}_{\proeet,c}(Y,\Q_p) & :=(\R\lim_n\R\Gamma_{\proeet,c}(Y,\Z/p^n))\otimes_{\Z_p}^{\LL_{\Box}}\Q_p;\\
 \text{or }\R\Gamma^{(2)}_{\proeet,c}(Y,\Q_p) &: =[\R\Gamma_{\proeet}(Y,\Q_p)\to \R\Gamma_{\proeet}(\partial Y,\Q_p)],\notag\\
   & \hphantom{aaa} \text{with } \R\Gamma(\partial Y,\Q_p) =\colim_{Z\in\Phi_Y}\R\Gamma(Y\setminus Z,\Q_p),\notag
 \end{align}
 where $\Phi_Y$ denotes the set of quasi-compact opens in $Y$ and    the brackets $[...]$ denote the mapping fiber.
 The first definition  is just a continuous version of $\Z/p^n$-cohomology.
  The heuristic for the second definition comes from the (possible) comparison with de Rham cohomology and the classical definition of van der Put of compactly supported version of the latter.
 Surprisingly,   this definition agrees with Huber's definition of compactly supported $p$-adic \'etale (sic !) cohomology: this follows from the  canonical quasi-isomorphism
  $$\R\Gamma^{(2)}_{\proeet,c}(Y,\Z_p)_{\Q_p}\stackrel{\sim}{\to}\R\Gamma^{(2)}_{\proeet,c}(Y,\Q_p).$$ The fact that this might be the right definition comes from the work on dualities which pairs this cohomology with $p$-adic pro-\'etale cohomology (see \cite{CGN}, \cite{CGN2}, \cite{ZL}, \cite{ALBM}).
 Hence we set
$$
\R\Gamma_{\proeet,c}(Y,\Q_p) :=\R\Gamma^{(2)}_{\proeet,c}(Y,\Q_p).
$$
\begin{theorem} \label{main-kg} 
{\rm (Basic comparison theorem, \cite[Cor.\,1.6]{AGN})} 
The analogs of Theorem \ref{basic1} and Theorem \ref{kol11}
hold for the compactly supported $p$-adic pro-\'etale cohomology. 
\end{theorem}

   \begin{example}
Let $Y:={\mathbb A}_K^d$, $d\geq 1$, be the rigid analytic affine space of dimension $d$ over $K$. 
\begin{enumerate}[leftmargin=*]
\item 
 We have the following isomorphisms  in solid $\Q_{p}$-modules
\begin{align*}
&  H^{i}_{\proeet,c}(Y_C,\Q_p) =0, \quad {\text{if $i<d$ or $i>2d$;}}\\
& H^{i}_{\proeet,c}(Y_C,\Q_p(i-d)) \simeq 
 H_c^d(Y_C, \Omega^{i-d-1})/\kker d,\quad \mbox{if }  d\le i \le 2d-1.
\end{align*}
\item  We have an exact sequence  in solid $\Q_{p}$-modules
$$ 0 \to H_c^d(Y_C, \Omega^{d-1})/ \kker d\to H_{\proeet,c}^{2d}(Y_C, \Q_p(d)) \to \Q_p \to 0. 
$$
\end{enumerate}
For $d=1$ and $i\geq 0$,   this yields:

$\bullet$ vanishings $H^i_{\proeet,c}({\mathbb A}_C^1, \Q_p)=0$, for $i\neq 2$,

$\bullet$ an exact sequence of solid $\Q_p$-modules
   \[ 0 \to H^1_c({\mathbb A}_C^1,\so) \to H^2_{\proeet,c}({\mathbb A}_C^1, \Q_p(1)) \to \Q_p\to 0. \] 
   \end{example}

\subsubsection{The arithmetic case}\label{arit1}
 The reader might wonder  what can we say about the $p$-adic pro-\'etale cohomology  of  a smooth dagger variety $Y$  itself (still defined over $K$).
 It turns out that in this case we do have an analog of the syntomic-pro-\'etale comparison \eqref{BK1} 
with the arithmetic syntomic cohomology 
$$\R\Gamma^{\rm BK}_{\synt}(Y):=
\big[[\R\Gamma_{\hk}(Y)]^{N=0,\phi=p^r}\stackrel{\iota_{\hk}}{\longrightarrow}\R\Gamma_{\dr}(Y)/{\rm Fil}^r\big],
$$ where
 $\R\Gamma_{\hk}(Y)$ is the arithmetic version of Hyodo-Kato cohomology. 
The existence of such an analog was not obvious at all to experts 
as the original methods (Hyodo, Kato, Tsuji) of proving Theorem \ref{kol11} in the algebraic case via the Milnor symbols were not easily adaptable to the arithmetic setting\footnote{One of us tried very hard to make this approach work but the torsion that appears in crystalline
cohomology of ramified extension was very difficult to deal with.}.
 It is here that the $(\phi,\Gamma)$-modules approach\footnote{The starting point of our investigations was the
observation that the syntomic and $(\varphi,\Gamma)$-module approaches to
compute $H^\bullet(G_F,\Q_p(r))$ (for $r\geq 1$ and $F$ a finite, unramified,
extension of $\Q_p$) produced complexes that look so similar that there ought to be a way to go from
one to the other via simple quasi-isomorphisms. Also, the syntomic approach proved to be much
easier to use to do actual computations (but it only applies to very special coefficients -- 
though more general than $\Q_p(r)$ with $r\geq 1$, see~\cite{abhi}).} 
showed its superiority (see \cite{AI}
for basic facts about relative $(\phi,\Gamma)$-modules).

  The analog of \eqref{BK1}
gives  the following:
\begin{theorem} \label{basicA}
{\rm (Basic arithmetic comparison theorem \cite{CN1,CN3})}
Let $n\leq r$ be positive integers. We have a long exact sequence in $\sd(\Q_{p,\Box})$
\begin{equation}\label{basic2a}
\cdots\to X^{r,n-1}\to {\rm DR}^{r,n-1}\to H^n_{\proeet}(Y,\Q_p(r))\to
X^{r,n}\to {\rm DR}^{r,n}\to\cdots
\end{equation}
where
$$X^{r,i}:=H^i([\R\Gamma_{\rm HK}(Y)]^{N=0,\varphi=p^r}),
\quad
{\rm DR}^{r,i}:=H^i(\rg_{\rm dR}(Y)/{\rm Fil}^r).
$$
\end{theorem}
We note  that in general $X^{r,i}$ is not equal to $H^i_{\hk}(Y)^{N=0,\phi=p^r}$ 
(see \eqref{paa2} below).  
The group ${\rm DR}^{r,i}$ tends to be  huge: in the case $Y$ is Stein we have
$$
{\rm DR}^{r,r-1}\simeq \Omega^{r-1}(Y)/{\rm Im}\,d;\quad {\rm DR}^{r,i}=0, \quad i\geq r.
$$

Theorem \ref{basicA} yields the following computation (\cite[Cor.\,1.4]{CN1},\cite[Th.\,1.3]{CN3}):
  \begin{enumerate}
    \item 
    For $1\leq i\leq r-1$, the boundary map induced by the sequence  (\ref{basic2a})
    $$
    \partial_{r}: {H}^{i-1}_{\dr}(Y)\to {H}^{i}_{\proeet}(Y,\Q_p(r))
    $$
    is an isomorphism. In particular, the cohomology $ {H}^{i}_{\proeet}(Y,\Q_p(r))$  has a natural $K$-structure. 

    \item We have exact sequences
    \begin{align}\label{paa2}
   &  0\to {\rm DR}^{r-1,r}\lomapr{\partial_r} {H}^r_{\proeet}(Y,\Q_p(r))\to X^{r,r} \to {\rm DR}^{r,r}\notag\\
   & 
0\to  {H}^{r-1}_{\hk}(Y)^{\phi=p^{r-1}}  \to X^{r,r}\to {H}^{r}_{\hk}(Y)^{N=0,\phi=p^r}\to 0
    \end{align}
\end{enumerate}
\begin{remark}
It follows form the above results that,
if $[K:\Q_p]<\infty$, and if $Y$ is an overconvergent affinoid 
(or, more generally, if $Y$ has finite dimensional de Rham cohomology),
then ${H}^{i}_{\proeet}(Y,\Q_p(r))$ is finite dimensional over $\Q_p$ if $0\leq i\leq r-1$.
This is in sharp contrast with what happens for $i=r$ where ${\rm DR}^{r,r-1}$ injects
into ${H}^{i}_{\proeet}(Y,\Q_p(r))$, and can be huge since we have the exact sequence
 $$
 0\to H^{r-1}_{\dr}(Y)\to {\rm DR}^{r,r-1}\to \ker \pi\to 0,
 $$
 where $\pi: \Omega^r(Y)^{d=0}\to H^r_{\dr}(Y)$ is the canonical map.
 \end{remark}

\section{Pro-\'etale cohomology of affinoids}
We work in this section in the $\infty$-derived category $\sd(C_{\Q_p})$ 
of locally convex topological vector spaces over $\Q_p$ (this is less robust than the condensed
version for doing homological algebra, but is sufficient in many cases).

\subsection{Improving the comparison with syntomic cohomology}
The de Rham cohomology of affinoids is notoriously pathological (it is infinite dimensional
and non separated); this is the reason for making them overconvergent
and defining dagger varieties. On the other hand, if $Y$ is an affinoid over $K$,
its pro-\'etale cohomology is much better behaved (at least if it has good reduction):
$H^r_{\proeet}(Y_C,\Q_p(r))$ is a Banach space (an elementary
fact in dimension~$1$; the general case was proven by Bosco~\cite{GB3}).


Now, the methods used to prove Theorem \ref{basic1} yield an exact sequence\footnote{We use here the $\breve{C}$-version of Hyodo-Kato cohomology.} in ${\rm LH}(C_{\Q_p})$ (the left heart of $\sd(C_{\Q_p})$)
$$0\to \Omega^{r-1}(Y_C)/{\rm Ker}\,d\to H^r_{\proeet}(Y_C,\Q_p(r))
\to (\bst^+\wotimes_{\breve{C}} H^r_{\rm HK}(Y_C))^{N=0,\varphi=p^r}\to 0$$
This is not very satisfactory since 
$H^r_{\rm HK}(Y_C)$ is not separated (because $H^r_{\rm dR}(Y_C)$ is not separated as we mentioned above,
and we have the Hyodo-Kato isomorphism relating the two), 
which makes the right-hand term not separated.
We denote by $H^1_{\rm HK}(Y_C)^{\rm sep}$ its quotient by the topological closure of $0$;
this is a finite dimensional $\breve C$-module, and we still have a Hyodo-Kato isomorphism
$$C\otimes_{\breve C}H^1_{\rm HK}(Y_C)^{\rm sep}\simeq H^1_{\rm dR}(Y_C)^{\rm sep}.
$$

If $Y$ be an affinoid over $C$, let $\O(Y)^{\dual\dual}$ be the group of
$f\in\O(Y)^\dual$ such that $f-1$ is topologically nilpotent.
The following theorem fixes the above mentioned problem in dimension~$1$.
\begin{theorem}\phantomsection\label{intro1.1}
{\rm (\cite[Th.\,0.1]{CDN2})}
If $Y$ is an affinoid of dimension $1$,
we have a commutative diagram of 
Banach spaces \footnote{If $M$ is a $\Z$-module,
we set $\Q_p\wotimes M:=\Q_p\otimes_{\Z_p}(\varprojlim_n M/p^nM)$.}
$$
\xymatrix@R=.6cm@C=.5cm{
0\ar[r] &\Q_p\wotimes \O(Y)^{\dual\dual}
\ar[r]\ar@{=}[d] & H^1_{\proeet}(Y,\Q_p(1))\ar[d] \ar[r] &
(\bst^+\otimes_{\breve C} H^1_{\rm HK}(Y)^{\rm sep})^{N=0,\varphi=p}\ar[r]\ar[d]^-{\theta\otimes\iota_{\rm HK}} & 0\\
0\ar[r]& \Q_p\wotimes \O(Y)^{\dual\dual} \ar[r]^-{\rm dlog}
& \Omega^1(Y) \ar[r] & H^1_{\rm dR}(Y)^{\rm sep}\ar[r] & 0
} $$
with upper row being exact and lower row being a complex
and all arrows having closed image.
\end{theorem}

\begin{remark}\label{hk1}
The group $H^1_{\rm HK}(Y)^{\rm sep}$ has a concrete combinatorial description (see Th.\,\ref{expo2}),
 starting from
a triangulation of $Y$ (or, what amounts to the same, a semistable model).
It is easy to obtain analogous   results for Stein curves or overconvergent affinoids from the above result
and this combinatorial description~\cite[Th.\,0.5, Th.\,0.7]{CDN2}.
\end{remark}

\begin{remark}\phantomsection\label{basic42}
If $Y$ is an affinoid of dimension~$d$, and if $r\leq d$,
the above result suggests that one could have an exact sequence
$$0\to \Q_p\wotimes K^r_{\cal M}(Y)^{++}\to H^r_{\proeet}(Y,\Q_p(r))
\to (\bst^+\otimes_{\breve{C}} H^r_{\rm HK}(Y)^{\rm sep})^{N=0,\varphi=p^r}\to 0$$
where $K^r_{\cal M}(Y)^{++}$ is the subgroup of the Milnor
$K$-theory group $K^r_{\cal M}(\O(Y))$ generated by symbols
$(f_1,\dots,f_r)$, with  $f_i\in\O(Y)^{\dual\dual}$ (this restriction
makes
Steinberg relation disappear since you cannot have 
$x$ and $1-x$ both belonging to $\O(Y)^{\dual\dual}$).
\end{remark}

The next two sections are devoted to the ingredients that come into play
for the proof of Theorem~\ref{intro1.1}. The general philosophy is inspired by
the proof of Theorem~\ref{basic1}, but working with curves allows some notable simplifications
(in particular, the construction of the Hyodo-Kato isomorphism becomes straightforward).

\subsection{Chopping a curve into shorts and legs}\label{intro6.1}
Let $Y$ be a quasi-compact curve over $C$. Choose a triangulation $S$ of $Y$ and
denote by $Y_S$ the associated semistable model of $Y$.
We can (and will) choose $S$ fine enough so that the irreducible components
of the special fiber $Y_S^{\rm sp}$ are smooth and intersect in at most one point.

We see $Y_S^{\rm sp}$ 
as a {\it proper curve over $k_C$ endowed with a set $A$ of marked points}
$a=(P_a,\mu(a))$, with $A=A_c\sqcup (A\moins A_c)$, and:

$\bullet$ $A_c$ is the set of singular points
and $\mu(a)\in\Q_{>0}$ if $a\in A_c$;

$\bullet$ $A\moins A_c$ is the set of points that one has to remove to get the
usual special fiber and $\mu(a)=0^+$ if $a\in A\moins A_c$.

The irreducible components of $Y_S^{\rm sp}$ are smooth and proper and in bijection with $S$
(we denote by
$ Y^{\rm sp}_s$ the component
corresponding to $s\in S$).

\vskip2mm
From this data, one builds {\it the adoc graph}
$\Gamma^{\rm ad}(Y_S)$ of $Y_S$, whose vertices are $S$, edges $A$, each edge $a$ being of length
$\mu(a)$ with end points the $s\in S$ such that $P_a\in Y^{\rm sp}_s$
(hence $a\in A_c$ has two end points whereas $a\in A\moins A_c$ has only one).
The graph $\Gamma^{\rm ad}(Y_S)$ is thus the dual graph of the classical (non proper)
special fiber with added edges of length $0^+$ (i.e., infinitesimal edges) originating from vertices corresponding
to nonproper irreducible components, one such edge for each missing point (this graph can also
be seen as the dual graph attached to a log compactification of the classical special fiber).

If $s\in S$, the tube $Y_s$ of $Y^{\rm sp}_s$ with marked points removed is
{\it a shorts} (i.e.~(the formal model of) an affinoid
with good reduction); if $a\in A_c$, the tube $Y_a$ of $P_a$ is {\it a leg} (i.e., an open annulus)
of length $\mu(a)$ (one has $\O(Y_a)=\O_C[[T_{a,s_1},T_{a,s_2}]]/(T_{a,s_1}T_{a,s_2}-p^{\mu(a)})$,
if $s_1,s_2$ are the end points of~$a$).

The $Y_i$, for $i\in I=S\sqcup A_c$,
form a partition of $Y$. To reconstruct $Y$, one needs an additional gluing data
for $Y_s$ and $Y_a$ if $(a,s)\in I_{2,c}$,  where
$I_{2,c}$ is the set of $(a,s)$ with $a\in A_c$
and $s$ an end point of $a$ (such a gluing does not make sense
in rigid or adic geometry, but it does in {\it adoc}\footnote{An interpolation between ``ad hoc''
and ``adic''.} {\it geometry}, see~\cite[chap.\,3]{CDN2}):
$P_a$ determines a rank $2$ valuation on $\O(Y_s)$, hence 
{\it a ghost\footnote{It has no point in characteristic $0$, hence the name; on the other hand, its
reduction is the punctured open ball in characteristic $p$.} circle} $Y_{s,a}$ (choosing a local parameter gives an 
isomorphism $\O(Y_{s,a})\simeq \O_C[[T_{s,a},T_{s,a}^{-1}\rangle$,
completion of $\O_C[[T_{s,a}]][T_{s,a}^{-1}]$ for the $p$-adic topology),
and $Y_a$ contains a corresponding ghost circle $Y_{a,s}$; the gluing data is then
an isomorphism $\iota_{a,s}:Y_{a,s}\simeq Y_{s,a}$, i.e.~an isomorphism
$\O_C[[T_{s,a},T_{s,a}^{-1}\rangle\simeq \O_C[[T_{a,s},T_{a,s}^{-1}\rangle$.

The above set of data
(i.e.~$\Gamma= (S,A, A_c,\mu)$, $(Y_i)_{i\in I}$, $(\iota_{i,j})_{(i,j)\in I_{2,c}}$),
is {\it a pattern of curve}, see~\cite[\S\S\,3.5,\,3.6]{CDN2}.  The above discussion explains
how to associate a pattern of curve to a triangulation, and how to reconstruct
$Y$ from this pattern; we refer to the figure below for a visual illustration of this procedure.
\begin{center}
\begin{tikzpicture}
[scale=.8]
\draw[very thick] (0,4) arc (50: 310 : 1.4 and 3);
\draw[very thick] (5,4) arc (130: -130 : 1.4 and 3);

\draw[very thick]  (0,.2) ..controls +(1,-.1) and +(-1,-.1).. (5,.2);
\draw[very thick]  (0,-.6) ..controls +(1,.1) and +(-1,.1).. (5,-.6);
\draw[very thick, dotted, red] (0,-.6) arc (-90: 90 : .2 and .4); \draw[very thick, dashed, red] (0,.2) arc (90: 270 : .2 and .4);
\draw[very thick, dotted, red] (5,-.6) arc (-90: 90 : .2 and .4); \draw[very thick, dashed, red] (5,.2) arc (90: 270 : .2 and .4);

\draw[very thick]  (0,4) ..controls +(1,-.1) and +(-1,-.1).. (1.8,4);
\draw[very thick]  (0,3.2) ..controls +(1,.1) and +(-1,.1).. (1.8,3.2);
\draw[very thick, dotted, red] (0,3.2) arc (-90: 90 : .2 and .4); \draw[very thick, dashed, red] (0,4) arc (90: 270 : .2 and .4);
\draw[very thick, dotted, red] (1.8,3.2) arc (-90: 90 : .2 and .4); \draw[very thick, dashed, red] (1.8,4) arc (90: 270 : .2 and .4);
\draw[very thick, dotted, red] (3.2,3.2) arc (-90: 90 : .2 and .4); \draw[very thick, dashed, red] (3.2,4) arc (90: 270 : .2 and .4);
\draw[very thick, dotted, red] (5,3.2) arc (-90: 90 : .2 and .4); \draw[very thick, dashed, red] (5,4) arc (90: 270 : .2 and .4);
\draw[very thick]  (5,4) ..controls +(-1,-.1) and +(1,-.1).. (3.2,4);
\draw[very thick]  (5,3.2) ..controls +(-1,.1) and +(1,.1).. (3.2,3.2);

\draw[very thick] (1.8,3.2) arc (-150:-30:.8); \draw[very thick] (1.8,4) arc (150:30:.8);

\draw[very thick]  (5,.2) ..controls +(.5,0) and +(.5,0).. (5,3.2);
\draw[very thick]  (0,.2) ..controls +(-.5,0) and +(-.5,0).. (0,3.2);

\draw[thick, blue] (-1,-.85) circle(.25);
\draw[thick, blue] (6,-.85) circle(.25);
\draw[blue] (-1,-3) -- (-1,-.7);
\draw[blue] (-1,-1.5) -- (-.9,-.9);
\draw[ball color=blue] (9.5,-.85) circle (.12);
\draw[ball color=blue] (11.5,-.85) circle (.12);
\draw[blue](6,-.6)--(9.5,-.85);
\draw[blue](6,-1.1)--(9.5,-.85);

\draw[thick, blue] (-1,1.75) circle(.25);
\draw[ball color=blue] (8.85,1.75) circle (.12);
\draw[blue](-1,2)--(8.85,1.75);
\draw[blue](-1,1.5)--(8.85,1.75);

\draw[very thick]  (-1.5,-.2) ..controls +(0.15,-.35) and +(-0.15,-.35).. (-.5,-.2);
\draw[very thick]  (-1.375,-.325) ..controls +(.1,.15) and +(-.1,.15).. (-.625,-.325);

\draw[very thick]  (-1.7,1) ..controls +(0.15,-.35) and +(-0.15,-.35).. (-.7,1);
\draw[very thick]  (-1.575,.875) ..controls +(.1,.15) and +(-.1,.15).. (-.825,.875);
\draw[very thick]  (-1.7,2.5) ..controls +(0.15,-.35) and +(-0.15,-.35).. (-.7,2.5);
\draw[very thick]  (-1.575,2.375) ..controls +(.1,.15) and +(-.1,.15).. (-.825,2.375);

\draw[very thick]  (-1.5,3.8) ..controls +(0.15,-.35) and +(-0.15,-.35).. (-.5,3.8);
\draw[very thick]  (-1.375,3.675) ..controls +(.1,.15) and +(-.1,.15).. (-.625,3.675);

\draw[very thick]  (5.5,3.8) ..controls +(0.15,-.35) and +(-0.15,-.35).. (6.5,3.8);
\draw[very thick]  (5.625,3.675) ..controls +(.1,.15) and +(-.1,.15).. (6.375,3.675);

\draw[very thick]  (5.7,1.6) ..controls +(0.15,-.35) and +(-0.15,-.35).. (6.7,1.6);
\draw[very thick]  (5.825,1.475) ..controls +(.1,.15) and +(-.1,.15).. (6.575,1.475);

\draw[very thick]  (5.5,-.2) ..controls +(0.15,-.35) and +(-0.15,-.35).. (6.5,-.2);
\draw[very thick]  (5.625,-.325) ..controls +(.1,.15) and +(-.1,.15).. (6.375,-.325);

\draw[very thick, dotted, color=red, fill=gray!20] (6.7,2.8) circle (.27 and .3);
\draw[thick, dotted, red] (6.7,3.1) -- (12.5,2.6);
\draw[thick, dotted, red] (6.7,2.5) -- (12.5,2.6);
\draw[very thick, dotted, color=red, fill=gray!20] (6.7,.8) circle (.27 and .3);

\draw(2.5,0.45) node{$Y_{a_3}$};
\draw(1,4.3) node{$Y_{a_1}$};
\draw(4,4.3) node{$Y_{a_2}$};
\draw(2.5,4.7) node{$Y_{s_2}$};
\draw(-1.9,4.5) node{$Y_{s_1}$};
\draw(6.9,4.5) node{$Y_{s_3}$};

\draw(2.5,-3.4) node{$\Gamma^{\rm ad}(Y)$};
\draw[ball color=black] (-1,-3) circle (.2);
\draw[ball color=black] (6,-3) circle (.2);
\draw[ball color=black] (2.5,-2) circle (.2);
\draw[very thick] (-1,-3)--(6,-3);
\draw[very thick] (-1,-3)--(2.5,-2);
\draw[very thick] (2.5,-2)--(6,-3);
\draw (.75,-2.2) node{$a_1$};
\draw (4.25,-2.2) node{$a_2$};
\draw (2.5,-2.75) node{$a_3$};
\draw (6.2,-2.2) node{$a_4$};
\draw (6.9,-2.7) node{$a_5$};
\draw (-1,-3.4) node{$s_1$};
\draw (2.5,-1.6) node{$s_2$};
\draw (6,-3.4) node{$s_3$};
\draw[ball color=blue] (3.5,-3) circle (.12);
\draw[very thick, dashed, blue] (3.5,-.55) arc (-90: 90 : .2 and .35); \draw[very thick, blue] (3.5,.15) arc (90: 270 : .2 and .35);
\draw[blue] (3.5,-3) -- (3.35,0);
\draw[thick,color=red,{->}] (6,-3) -- +(80:.7); \draw[thick,color=red,{->}] (6,-3) -- +(20:.7);
\draw[thick,color=red,{->}] (6,-3) -- (5.5,-2.86); \draw[thick,color=red,{->}] (6,-3) -- (5.5,-3);
\draw[thick,color=red,{->}] (-1,-3) -- (-.5,-2.86); \draw[thick,color=red,{->}] (-1,-3) -- (-.5,-3);
\draw[thick,color=red,{->}] (2.5,-2) -- (2,-2.14); \draw[thick,color=red,{->}] (2.5,-2) -- (3,-2.14);

\draw[very thick]  (8,4.5) ..controls +(.5,-3) and +(-2,.5).. (12,-1);
\draw[very thick]  (13,4.5) ..controls +(-.5,-3) and +(2,.5).. (9,-1);
\draw[very thick]  (7.5,3.5) -- (13.5,3.5);
\draw[ball color=red] (10.5,-.24) circle (.12);
\draw (11.6,-.24) node{$(P_{a_3},\mu_3)$};
\draw[ball color=red] (8.2,3.5) circle (.12);
\draw (9.2,3.2) node{$(P_{a_1},\mu_1)$};
\draw[ball color=red] (12.8,3.5) circle (.12);
\draw (13.7,3.2) node{$(P_{a_2},\mu_2)$};
\draw[ball color=red] (12.5,2.6) circle (.12);
\draw (13.5,2.4) node{$(P_{a_4},0^+)$};
\draw[ball color=red] (11.55,.8) circle (.12);
\draw (12.5,.6) node{$(P_{a_5},0^+)$};
\draw(10.5,3.9) node{$Y_{s_2}^{\rm sp}$};
\draw(8.45,4.3) node{$Y_{s_1}^{\rm sp}$};
\draw(12.55,4.3) node{$Y_{s_3}^{\rm sp}$};
\draw(10.5,-1.2) node{$Y^{\rm sp}$};
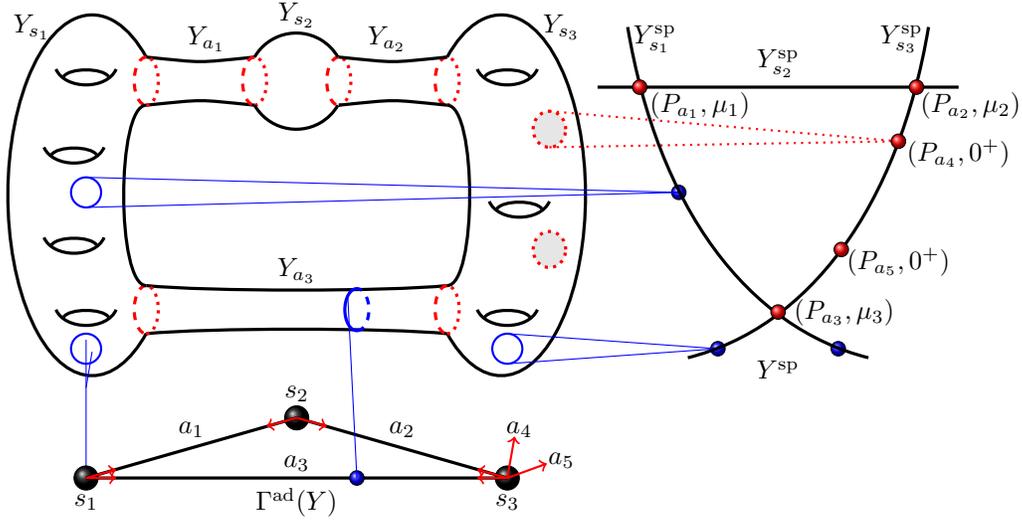
\draw (5.5,-4.4) node{An affinoid $Y$ built from 3 shorts and 3 legs};

\end{tikzpicture}
\end{center}
{\Small
The above affinoid $Y$ is obtained by gluing, along ghost circles
(red dots), shorts $Y_{s_1}$, $Y_{s_2}$, $Y_{s_3}$
of respective genus $4$, $0$ et $3$, and legs $Y_{a_1}$, $Y_{a_2}$, $Y_{a_3}$ of respective length
 $\mu_1$, $\mu_2$, $\mu_3$.
The special fiber $Y^{\rm sp}$ 
has five marked points: three singular corresponding to the specializations of the legs
and two to the tubes removed from $Y_{s_3}$.  The skeleton $\Gamma^{\rm an}(Y)$ is the
triangle with vertices $s_1$, $s_2$ et $s_3$ (the length of edges $a_1$, $a_2$ and $a_3$
are $\mu_1$, $\mu_2$ and  $\mu_3$) and $\Gamma^{\rm ad}(Y)$ is obtained by adding
the arrows $a_4$ and  $a_5$ (of length $0^+$).


}

\vskip2mm
Let $S_{\rm int}\subset S$, be the set of $s$ such that $Y_s^{\rm sp}$ has no marked point
$(a,\mu(a))$ with $\mu(a)=0^+$, and let $\Gamma_{\rm int}$ be the full subgraph of
$\Gamma^{\rm ad}(Y_S)$ with set of vertices $S_{\rm int}$ (this is {\it the interior}
of $\Gamma^{\rm ad}(Y_S)$).
\begin{theorem}\phantomsection\label{expo2}
%
{\rm (\cite[Th.\,0.13]{CDN2})}.
$H^1_{\rm HK}(Y)^{\rm sep}$ admits a natural filtration whose
successive quotients are\footnote{\label{mono}
If $\Gamma$ is a graph, we let $H^\bullet(\Gamma,\Lambda)$
and $H^\bullet_c(\Gamma,\Lambda)$ denote its cohomology or cohomology with compact
support.  These groups have a natural combinatorial description in terms of
operators $\partial$ and $\partial^\dual$ for which we refer to~\cite[\S\,1.2]{CDN2}.
If $\Gamma$ is metrized (as is the case of $\Gamma^{\rm ad}(Y_S)$), this induces
$N:H^1_c(\Gamma,\Lambda)^\dual\to H^1(\Gamma,\Lambda)$ involving the length of the edges, 
see~\cite[1.2.4]{CDN2}.}
$$H^1_{\rm HK}(Y)^{\rm sep}=\big[
\xymatrix@C=.3cm{
H^1(\Gamma_{\rm int},\breve C)\ar@{-}[r]&\prod\limits_{s\in S_{\rm int}}\hskip-.1cm H^1_{\rm rig}( Y^{\rm sp}_s) 
\oplus\prod\limits_{s\in \partial Y}\hskip-.1cm H^1_{\rm rig}( Y^{\rm sp}_s)^{[1]}
\ar@{-}[r]& H_c^1(\Gamma,\breve C)^\dual(-1)}\big],$$
where the exponent $[1]$ denotes the slope~$1$ subspace {\rm (for $\varphi$)},
and the $(-1)$ twists means that the natural action of $\varphi$ is multiplied by $p$.
\end{theorem}
\begin{remark}\label{hk2}
The operator $N$ of footnote~\ref{mono} induces $N:H_c^1(\Gamma,\breve C)^\dual
\to H^1(\Gamma,\breve C)\to H^1(\Gamma_{\rm int},\breve C)$ (the second map being
restriction), and hence a
``monodromy operator'' $N:H^1_{\rm HK}(Y)^{\rm sep}\to H^1_{\rm HK}(Y)^{\rm sep}$.
\end{remark}

\subsection{Proofs}
Choose a fine enough triangulation $S$ of $Y$, let $Y_S$ be the
associated semistable model, and let 
$(\Gamma,(Y_i)_{i\in I},(\iota_{i,j})_{(i,j)\in I_{2,c}})$ be
the associated pattern.
The proof of Th.\,\ref{intro1.1} relies on:

$\bullet$ the definition of a group of symbols
${\rm Symb}_p(Y)$, 

$\bullet$ the construction of \'etale and syntomic regulators inducing isomorphisms:
$$H^1_{\eet}(Y,\Z_p(1))\overset{\sim}\leftarrow {\rm Symb}_p(Y)\overset{\sim}{\rightarrow}
H^1_{\rm syn}(Y_S,\Z_p(1)),$$

$\bullet$
a description of $H^1_{\rm syn}(Y_S,\Z_p(1))$ in terms of the de Rham cohomology and its variations.
Only this part introduces denominators (these depend on the length of the legs: the smaller
the legs, the bigger the denominators).
\subsubsection{Symbols}
Let $Y$ be a quasi-compact curve over $C$, let $C(Y)$ be the field of meromorphic functions on $Y$, 
and let ${\rm Div}(Y)$ be the group of finite sums $\sum_{x\in Y(C)}n_xx$, with $n_x\in \Z$
for all $x$.

One defines {\it the symbol group} ${\rm Symb}_p(Y)$
to be
$${\rm Symb}_p(Y)=
\frac{\{(f_n)_{n\in\N},\ f_n\in C(Y)^\dual,\ {\rm Div}(f_n)\in p^n{\rm Div}(Y),
\ f_{n+1}/f_n\in(C(Y)^\dual)^{p^n}\}}{\{(h_n^{p^n}),\ h_n\in C(Y)^\dual\}}.$$
We have an exact sequence
$$0\to H^1(\Gamma,\Z_p(1))\to 
{\rm Symb}_p(Y)\to {\rm Ker}\big(\prod_{i\in I}{\rm Symb}_p(Y^{\rm gen}_i)\to 
\prod_{(i,j)\in I_{2,c}}{\rm Symb}_p(Y^{\rm gen}_{i,j})\big),$$
where ``gen'' stand for ``generic fiber'' and
${\rm Symb}_p(Y^{\rm gen}_i)$ is defined as above.

\subsubsection{\'Etale regulators}
Kummer theory gives an \'etale regulator
${\rm Symb}_p(Y)\to H^1_{\eet}(Y,\Z_p(1))$ which can be proven to be an isomorphisme
thanks to Kummer exact sequence
$$0\to (\Z/p^n)\otimes\O(Y)^\dual\to H^1_{\eet}(Y,\Z/p^n(1))\to {\rm Pic}(Y)[p^n]\to 0$$
Via the above exact sequence, this reduces the computation of
$H^1_{\eet}(Y,\Z_p(1))$ to that of the
${\rm Symb}_p(Y^{\rm gen}_i)$'s.

\subsubsection{Syntomic regulators}
Let\footnote{$\widetilde Y_S$ is just a notation, not an actual object.}
$C_{\rm dR}(\widetilde Y_S)$ be the complex computing
the absolute log-crystalline cohomology $H^\bullet_{\rm dR}(\widetilde Y_S)$ of $Y_S$; i.e.,
$$C_{\rm dR}(\widetilde Y_S):=\xymatrix@C=.6cm{\prod_{i\in I}\O(\widetilde Y_i)
\ar[r]& \prod_{i\in I}\Omega^1(\widetilde Y_i)
\oplus\prod_{(i,j)\in I_{2,c}}\O(\widetilde Y_{i,j})\ar[r]&
\prod_{(i,j)\in I_{2,c}}\Omega^1(\widetilde Y_{i,j})_{d=0}}$$
where the terms of the complex are defined as follows:

$\bullet$ A shorts
$Y_s$ is obtained by extension of scalars from a formal scheme
$\breve Y_s$ smooth over $\O_{\breve C}$, and we set
$\O(\widetilde Y_s)=\acris\wotimes_{\O_{\breve C}}\O(\breve Y_s)$ (and we choose
a Frobenius $\varphi$
on $\O(\breve Y_s)$).

$\bullet$ A leg
$Y_a$ is of the form $\O(Y_a)=\O_C[[T_1,T_2]]/(T_1T_2-p^r)$, 
where $r$ is the length of $Y_a$. We chose a morphism of semi-groups $r\mapsto \tilde p^r$,
from $\Q_+$ to $\acris$, where $\tilde p^r$
is a Teichm\"uller with $\theta(\tilde p^r)=p^r$,
and we set
$\O(\widetilde Y_a)=\acris[[T_1,T_2]]/(T_1T_2-\tilde p^r)$
(and we take $\varphi$ defined by $\varphi(T_i)=T_i^p$, for $i=1,2$).

$\bullet$ If $(a,s)\in I_{2,c}$, we define $\O(\widetilde Y_{a,s})$ to be the
$p$-adic completion of the divided power envelop of $\O(\widetilde Y_a)\wotimes \O(\widetilde Y_s)$
with respect to the kernel of the natural map to $\O(Y_{a,s})$, endowed with the
tensor product $\varphi$.

\vskip1mm
We define the {\it syntomic cohomology groups} $H^i_{\rm Syn}(Y_S,\Z_p(1))$ of $Y_S$
as those of the complex 
$${\rm Syn}(Y_S,\Z_p(1)):=
[\xymatrix{F^1C_{\rm dR}(\widetilde Y_S)\ar[r]^-{1-\varphi/p}&C_{\rm dR}(\widetilde Y_S)}]$$

\vskip.1cm
One defines a {\it syntomic regulator} ${\rm Symb}_p(Y)\to H^1_{\rm syn}(Y_S,\Z_p(1))$
by sending $(f_n)_{n\in\N}$ to the cocycle
$\lim_{n\to\infty}\big(\frac{d\tilde f_{n,i}}{\tilde f_{n,i}},\frac{1}{p}\log\frac{\varphi(\tilde f_{n,i})}
{\tilde f_{n,i}^p},\log(\frac{\tilde f_{n,i}}{\tilde f_{n,j}})\big)$,
where $\tilde f_{n,i}\in\O(\widetilde Y_i)$ is a lifting of the restriction
of $f_n$ to $Y_i$.
This regulator is also ``almost''
an isomorphism (on the nose if $p\neq 2$, with a cokernel killed by $2$ if $p=2$).
This is proven by gluing the corresponding isomorphisms for the $Y_i$'s, the case of shorts
\cite[Chap.\,5]{CDN2}
being much more delicate than that of legs~\cite[4.2.4]{CDN2}.

\subsubsection{Syntomic cohomology and Hyodo-Kato cohomology}
The de Rham cohomology groups $H^i_{\rm dR}(Y)$ of $Y$ are those of
$C_{\rm dR}(Y)=C\otimes_{\acris}C_{\rm dR}(\widetilde Y_S)$
(because $C_{\rm dR}(Y_i)=C\otimes_{\acris}C_{\rm dR}(\widetilde Y_i)$ if $i\in I$,
and $C\otimes_{\acris}\big[\O(\widetilde Y_{i,j})\to \Omega^1(\widetilde Y_{i,j})_{d=0}\big]$
is quasi-isomorphic to $C_{\rm dR}(Y_{i,j})$, if $(i,j)\in I_{2,c}$, thanks to
Poincar\'e Lemma),
whereas its Hyodo-Kato cohomology groups
$H^i_{\rm HK}(Y)$ are, by definition (see~\cite{CDN2}, just before Th.\,6.21),
those of
$C_{\rm dR}(\breve Y_S):=\breve C\otimes_{\acris}C_{\rm dR}(\widetilde Y_S)$
(the extensions of scalars are through 
the surjective morphisms $\theta:\acris\to\O_C$ et $\theta_0:\acris\to\O_{\breve C}$).

   One can modify slightly (see~\cite[6.3.5]{CDN2})
$C_{\rm dR}(\widetilde Y_S)$ to obtain a quasi-isomorphic complex
$\overline C_{\rm dR}(\widetilde Y_S)$ such that
$\O_{\breve C}\hookrightarrow\acris$
induces a 
morphism of complexes
$\overline C_{\rm dR}(\breve Y_S)\to \Q_p\otimes\overline C_{\rm dR}(\widetilde Y_S)$,
which commutes with $\varphi$, and is a section of
$\Q_p\otimes\overline C_{\rm dR}(\widetilde Y_S)
\to \overline C_{\rm dR}(\breve Y_S)$. This gives (almost for free) isomorphisms
$$\Q_p\otimes H^1_{\rm dR}(\widetilde Y_S)\simeq \bcris^+\wotimes_{\breve C}H^1_{\rm HK}(Y)
\quad{\rm and}\quad
\iota_{\rm HK}:H^1_{\rm dR}(Y)\simeq C\wotimes_{\breve C}H^1_{\rm HK}(Y).$$
Now, playing with the different presentations of the mapping fiber
defining ${\rm Syn}(Y_S,\Z_p(1))$ and using the vanishing of $H^1(Y,\O)$, one
gets an exact sequence:
%
$$0\to \O(Y)/C\to \Q_p\otimes H^1_{\rm syn}(Y_S,\Z_p(1))
\to (\bcris^+\wotimes_{\breve C}H^1_{\rm HK}(Y))^{\varphi=p}\to 0$$
which is not independent of the choice of $r\mapsto\tilde p^r$.
To go from $H^1_{\rm HK}(Y)$ to $H^1_{\rm HK}(Y)^{\rm sep}$, one has to compare
$(\Q_p\wotimes\O(Y)^{\dual\dual})/\exp(\O(Y))$ and the closure of $0$
in $(\bcris^+\wotimes_{\breve C}H^1_{\rm HK}(Y))^{\varphi=p}$.
This comparison is already necessary to prove that the syntomic regulator
gives local isomorphisms
${\rm Symb}_p(Y_i^{\rm gen})\simeq H^1_{\rm syn}(Y_i,\Z_p(1))$, see~\cite[Prop.\,5.32]{CDN2} (the
most delicate result of that paper).

Finally, to
get something independent of the choice of $r\mapsto\tilde p^r$,
one has to 
use the Picard-Lefshetz formula
that describes the effect of changing
$r\mapsto \tilde p^r$ in terms of the monodromy operator $N$ on $H^1_{\rm HK}(Y)^{\rm sep}$ and
replace $\bcris^+\wotimes_{\breve C}H^1_{\rm HK}(Y)^{\rm sep}$
by $(\bst^+\wotimes_{\breve C}H^1_{\rm HK}(Y)^{\rm sep})^{N=0}$, see~\cite[6.4.4]{CDN2} for details.

\section{The ${\rm C}_{\rm st}$ conjecture}\label{nowy2} The basic comparison theorem (see Theorem \ref{basic1}) relates the derived $p$-adic pro-\'etale cohomology to the (derived) refined de Rham cohomology. For computations (see \cite{CDN1})  it turned out to be very useful to have also a nonderived version of this theorem. At the moment this exists only conjecturally in general but we do have  it in some very important special cases (like Stein varieties, for example). 
\subsection{Statement} Let $Y$ be a smooth, partially proper rigid analytic variety over $K$. Let $r\geq 0$. Recall that the Bloch-Kato 
 syntomic cohomology  fits into  the top distinguished triangle in the following diagram in $\sd(\Q_{p,\Box})$
$$\xymatrix@R=5mm@C=4mm{
\rg^{\rm BK}_{\rm syn}(Y_C,\Q_p(r))\ar[r]\ar[d] &[\rg_{\rm HK}(Y_C)\otimes^{\LL_{\Box}}_{C^{\nr}}\bst^+]^{N=0,\varphi=p^r}\ar[r]\ar[d]^-{\iota_{\hk}\otimes\iota}
&(\rg_{\rm dR}(Y)\otimes^{\LL_{\Box}}_K\bdr^+)/{\rm Fil}^r\ar@{=}[d]\\
{\rm Fil}^r(\R\Gamma_{\dr}(Y)\otimes^{\LL_{\Box}}_K\bdr^+)\ar[r] & \rg_{\rm dR}(Y)\otimes^{\LL_{\Box}}_K\bdr^+\ar[r] &(\rg_{\rm dR}(Y)\otimes^{\LL_{\Box}}_K\bdr^+)/{\rm Fil}^r
}$$
Clearly, the bottom row is also a distinguished triangle. 
The middle vertical map is induced by the Hyodo-Kato morphism $\iota_{\hk}$ and the
injection $\iota: \bst^+\hookrightarrow\bdr^+$.  For $n\leq r$, from the above diagram and the basic comparison theorem (see Theorem \ref{basic1}), we deduce a map 
of long exact sequences in $\sd(\Q_{p,\Box})$
$$\xymatrix@R=5mm@C=4mm{
\cdots\ar[r] & {\rm DR}^{r,n-1}\ar[r]\ar@{=}[d] &H^n_{\proeet}(Y_C,\Q_p(r))\ar[r]\ar[d]
& X^{r,n} \ar[r]\ar[d]^{\iota_{\hk}\otimes\iota} & {\rm DR}^{r,n}\ar[r]\ar@{=}[d] &\cdots\\
\cdots\ar[r] & {\rm DR}^{r,n-1}\ar[r] &H^n({\rm Fil}^r)\ar[r] 
& H^n_{\rm dR}(Y_K) \otimes_K^{{\Box}}\bdr^+\ar[r] & {\rm DR}^{r,n}\ar[r] &\cdots
}$$
Here (and below) we set ${\rm Fil}^r:={\rm Fil}^r(\R\Gamma_{\dr}(Y)\otimes^{\LL_{\Box}}_K\bdr^+)$. 

\begin{conjecture}\label{cst1}
 {\rm (The ${\rm C}_{\rm st}$-conjecture~\cite[Conj. 1.4, Conj. 1.6]{CN5})}

{\rm (i)} {\rm (de Rham to \'etale)} The middle square of the above diagram is bicartesian: i.e., we have an exact
sequence
$$0\to H^n_{\proeet}(X_C,\Q_p(r))\to H^n({\rm Fil}^r)\oplus X^{r,n}\to
H^n_{\rm dR}(X_K)\otimes^{\Box}_{K}\bdr^+\to 0$$

{\rm (ii)} {\rm (\'Etale to de Rham)} We can recover de Rham and Hyodo-Kato cohomologies from the pro-\'etale one\footnote{
Recall that $(\bdr^+[\tfrac{1}{t}])^{G_K}=K$ and $\varinjlim\nolimits_{[L:K]<\infty}
(\bst^+[\tfrac{1}{t}])^{G_K}=C^{\rm nr}$.}:
we have isomorphisms {\rm (of filtered $K$-modules for the first and of
$(\varphi,N,G_K)$-modules over $C^{\rm nr}$ for the second)}
\begin{align*}
{\rm Hom}_{G_K}(H^n_{\proeet}(Y_C,\Q_p),\bdr^+[\tfrac{1}{t}])&\simeq H^n_{\rm dR}(Y_K)^\dual\\
\varinjlim\nolimits_{[L:K]<\infty}
{\rm Hom}_{G_L}(H^n_{\proeet}(Y_C,\Q_p),\bst^+[\tfrac{1}{t}])&\simeq H^n_{\rm HK}(Y_C)^\dual
\end{align*}
\end{conjecture}
\begin{remark}
The above conjecture is an extension of a classical conjecture of Fontaine for \'etale cohomology
of algebraic varieties. Part (ii) is classically phrased in a covariant manner 
as in (iii) of Remark~\ref{ht4}.
This is equivalent to the above formulation because $H^n_{\eet}(Y_C,\Q_p)$ is finite dimensional in
the case of algebraic varieties, but a similar covariant formulation 
for nonproper analytic varieties does
not seem possible.
\end{remark}

\begin{theorem}\label{back1}
{\rm (i)}  In Conjecture \ref{cst1}, we have {\rm (i)$\Rightarrow$(ii)}.

{\rm (ii)} Conjecture \ref{cst1} is verified for proper varieties, Stein varieties,
analytification of algebraic varieties,
complements of a small tube of a smooth analytic subvariety of a proper variety,
products of finite number of varieties of the above type with finite dimensional
de Rham cohomology, products of   proper and  Stein varieties.
\end{theorem}

\begin{remark}
(i) Granting point (i), point (ii) corresponds to \cite[Th.\,8.1, Prop.\,8.17]{CN5}.
Point (i) is not explicitly stated in~\cite{CN5}, but it  is what is actually proved
in~\cite[Chap.\,9]{CN5}.

(ii)
In our announcement~\cite{COB1}, we claimed to be able to prove Conjecture \ref{cst1}
in full generality, but the reduction to the quasi-compact case that we had in mind run into
$\R\lim$ problems for coherent cohomology. Moreover,  our proof in the quasi-compact case relied on
a delicate induction on the number of affinoids needed to cover our space and one crucial lemma
turned out to be false. Hence we only have been able to verify the conjecture in broad classes
of examples but there is clearly room for improvements  (for example, removing the ``small''
from complements of small tubes or proper fibrations over Stein instead of products).
\end{remark}

\begin{question}\label{hk3}
If $Y$ is quasi-compact, can one compactify a model of $Y$ in such a way that the complement
of the special fiber of $Y$ is of codimension~$\geq 1$ ?  This is possible if $\dim Y=1$ and it
is also possible mod $p$. A positive answer would probably make it possible to prove
${\rm C}_{\rm st}$ for quasi-compact dagger varieties.
\end{question}

 \subsection{The case of Stein varieties} Conjecture  ${\rm C}_{\st}$ is true for Stein varieties but, in fact, in that case one can show more.  Let  $Y$ be a smooth  Stein variety over $K$. 
 From Theorem \ref{basic1}, one derives the following, very useful (see \cite{CDN1}), result:
 \begin{theorem}{\rm (Fundamental diagram, \cite[Th. 6.14]{CN5})}\label{affinoids}
 Let $r\geq 0$. There is  a natural map of strictly exact sequences in $\sd({\Q_{p,\Box}})$
  $$
\xymatrix@C=.6cm@R=.6cm{
0\to  \Omega^{r-1}(Y_C)/\kker d\ar[r]\ar@{=}[d] & H^r_{\proeet}(Y_C,\Q_p(r))\ar[d]^{\tilde{\beta}} \ar[r] & (H^r_{\hk}(Y_C){\otimes}^{\Box}_{C^{\nr}}\B^+_{\st})^{N=0,\phi=p^r}\ar[d]^{\iota_{\hk}\otimes\theta} \to 0\\
0\to  \Omega^{r-1}(Y_C)/\kker d \ar[r]^-d & \Omega^r(Y_C)^{d=0} \ar[r] & H^r_{\dr}(Y_C)\longrightarrow 0
}
$$
Moreover, 
$H^r_{\proeet}(Y_C,\Q_p(r))$ is Fr\'echet, 
the vertical maps are strict and  have closed images, and
$ \kker\tilde{\beta}\simeq  (H^r_{\hk}(Y_C){\otimes}^{\Box}_{C^{\nr}}\B^+_{\st})^{N=0,\phi=p^{r-1}}$. 
\end{theorem}
\subsection{Cohomology with compact support}
We can also formulate a ${\rm C}_{\rm st}$-conjecture for compactly supported cohomology
(at least an analog of part (ii) of Conjecture \ref{cst1}), but it takes a more sophisticated shape.
Instead of ${\rm Hom}_{G_K}(-,B)$ for $B=\bdr^+[1/t],\bst^+[1/t]$, we
need to take a derived $\Hom$ and also  kill Galois cohomology of period rings in degrees~$\geq 1$.
Otherwise one gets extra copies of the desired groups popping out in random degrees.
\subsubsection{Period ring $\B_{\rm pFF}$} To achieve the second goal, let us introduce new period rings. 

Define (as in~\cite{Csen}) the $p$-adic avatar $\log t$ of $\log 2\pi i$ 
as being transcendental over $\bdr^+$,
with $\sigma\in G_K$ acting via $\sigma(\log t)=\log t+\log\chi(\sigma)$ (remember that
$\sigma(t)=\chi(\sigma)t$).
Define (after Fontaine~\cite{Fpdr}) 
$$\bpdr^+:=\bdr^+[\log t]$$
with the filtration induced from $\bdr^+$. 

  Now, one can replace $\bst^+$ par $\brig^+[\log\tilde p]$, 
with $\brig^+:=\cap_{n\geq 0}\varphi^n(\bcris^+)$, in the statement of Conjecture \ref{cst1}
(because the relevant periods live in finite dimension spaces stable by $\varphi$).
But $\brig^+$ is the ring of analytic functions on the Fargues-Fontaine curve
${\rm Spa}\,\ainf\moins\{p=0\}$. For the same reason, one can also replace $\brig^+$
by the ring $\B$ of analytic functions on the Fargues-Fontaine curve
$Y_{\rm FF}={\rm Spa}\,\ainf\moins\{p=0,\tilde p=0\}$.  Define
$$\bFF^+:=\B[\log\tilde p]$$
and then $\bpFF^+$ as its unipotent envelope  (see \cite{CGN3} for details). 
By inverting $t$ we get the period rings $\bpdr, \bFF, \bpFF$.

  We have the following results:
\begin{theorem}\label{acyclicity}
{\rm (i)} {\rm (Semistable ver de Rham, \cite[Prop.\,10.3.15]{FF18}, \cite[Cor.\,2.3]{CGN3})} 
The natural morphism $\bFF^+\to\bdr^+$ is injective and
$G_K$-equivariant; the same is true for its natural extensions
$$\bpFF^+\to\bpdr^+,\quad \bpFF\to\bpdr.
$$

{\rm(ii)} {\rm (Galois acyclicity, \cite[Th.1.4, Th.\,2.4]{CGN3})} We have  the isomorphisms
\begin{align*}
 & H^0(G_K,\bpdr)=K,\quad 
H^0(G_K,\bpFF)=F, \\
& H^i(G_K,\Lambda)=0,\quad  \mbox{if } \,i\geq 1,  \quad \Lambda=\bpdr, \bpFF.
\end{align*}
\end{theorem}
\begin{remark} {\rm (Period sheaves as avatars of refined de Rham cohomologies)} Comparison quasi-isomorphisms from 
Corollary \ref{finish1} and  Theorem \ref{CGN21} combined with the Galois acyclicity from Theorem \ref{acyclicity} yield the following result (one of our original motivations for proving Theorem \ref{acyclicity}):
\begin{corollary}
\begin{enumerate}[leftmargin=*]
 \item {\rm (\cite[Cor.\,3.17]{DN})} Let  $Y$ be a smooth quasi-compact rigid analytic variety. There are natural quasi-isomorphisms in $\sd(K_{\Box})$ 
 \begin{align*}
F^0\rg_{\dr}(Y)\ & \simeq \rg(G_K, \rg_{\proeet}(Y_C,{\mathbb B}^+_{\rm pdR})),  \\
\rg_{\dr}(Y) & \simeq \rg(G_K, \rg_{\proeet}(Y_C,{\mathbb B}_{\rm pdR}) ).\notag
\end{align*}
\item  {\rm (\cite[Cor.\,3.31]{DN})} Assume that $k$ is algebraically closed. Let  $Y$ be  a smooth partially proper rigid analytic variety with finite rank de Rham cohomology. Then there is a  natural $G_K$-equivariant quasi-isomorphism in $\sd_{\phi,N}(\breve{C}_{\Box})$
\begin{align*}
 \R\Gamma_{\hk}(Y_C) \stackrel{\sim}{\to} ([\R\Gamma_{\proeet}(Y_C,{\mathbb B})[1/t]]^{\phi=1}\otimes^{\LL_{\Box}}_{\B_e}\B_{\rm pFF})^{\R G_K-{\rm sm}}.
\end{align*}
\end{enumerate}
\end{corollary}
\end{remark}

\subsubsection{Conjecture ${\rm C}_{\st}$} The new period rings described above allowed us to formulate the following conjecture:
\begin{conjecture}\label{cst2} 
Let $K$ be a finite extension of $\Q_p$. 
Let $Y$ be a smooth partially proper rigid analytic variety over $K$ of dimension $d$.
We can recover de Rham and Hyodo-Kato cohomologies from the compactly supported pro-\'etale one:
we have isomorphisms {\rm (of filtered $K$-modules for the first and of
$(\varphi,N,G_K)$-modules over $C^{\rm nr}$ for the second)}
\begin{align*}
{\rm RHom}_{G_K}(\rg_{{\proeet},c}(Y_C,\Q_p(d))[2d], \bpdr)&\simeq \rg_{\rm dR}(Y_K)\\
\varinjlim\nolimits_{[L:K]<\infty}
{\rm RHom}_{G_L}(\rg_{{\proeet},c}(Y_C,\Q_p(d))[2d], \bpFF)&\simeq \rg_{\rm HK}(Y_C)
\end{align*}
\end{conjecture}
\begin{remark}
(i) One can use Poincar\'e duality for de Rham cohomology to replace $H^n_{\rm dR}(Y_K)^\dual$
and $H^n_{\rm HK}(Y_C)^\dual$ by $H^{2d-n}_{\rm dR,c}(Y_K)$
and $H^{2d-n}_{\hk,c}(Y_C)$ in the statement of Conjecture \ref{cst1}, which gives it
a shape closer to that of Conjecture \ref{cst2}.

(ii) Implicit in the formulation of Conjecture \ref{cst2} is a choice of a suitable category of topological 
$\Q_p$-vector spaces with continuous action of $G_K$.
\end{remark}

  \subsubsection{Fundamental diagram}     
Let $Y$ be a smooth Stein rigid analytic  variety over $K$ of dimension $d$. In this case, under some restrictions, Theorem \ref{main-kg} yields  the following  result: 
 \begin{corollary} {\rm (Fundamental diagram, \cite[Cor.\,1.6]{AGN})}\label{affinoids-comp} 
Let $r\geq 0$. 
Then,  under a strong  assumption on the slopes of Frobenius on Hyodo-Kato cohomology (see \cite[Th. 8.4]{AGN} for details), there is a map of exact sequences of solid $\Q_p$-modules
{\small 
$$ \xymatrix@R=4mm@C=4mm{
0 \ar[r] & \big(H_c^d(Y_C, \Omega^{r-d-1})/ {\rm Ker}\, d\big)(d) \ar[r] \ar@{=}[d] & H_{\proeet,c}^{r}(Y_C, \Q_p(r)) \ar[d] \ar[r] & (H^{r}_{\hk,c}(Y_C) \otimes^{\Box}_{C^{\rm nr}} t^d \B_{\st}^+)^{\varphi=p^{r}, N=0} \ar[r] \ar[d] & 0 \\
0 \ar[r] & \big(H_c^d(Y_C, \Omega^{r-d-1})/ {\rm Ker}\, d\big)(d) \ar[r] & H_c^d(Y_C, \Omega^{r-d})^{d=0}(d)  \ar[r] & H^r_{\dr,c}(Y_C)(d) \ar[r] & 0 }
$$}
 \end{corollary}
The Frobenius slope condition above is very restrictive  but it holds for tori and Drinfeld spaces,
for example (probably also for general $p$-adic period domains, 
but not for the higher levels of the Drinfeld tower),
 so we can use the fundamental diagram to compute compactly supported cohomology in these cases.
 Recall that for the usual cohomology we have a fundamental diagram with no restrictions on the slopes.
 This means that the authors of \cite{CDN1} got very lucky that they started computations with the usual cohomology and not with the compactly supported one as is customary  in the local Langlands program.

\section{Geometrization of the basic comparison theorem}\label{geo1} 
 A crucial ingredient in the proof of claim (ii) of Theorem \ref{back1} 
 is the fact that the basic comparison
theorem (see Theorem \ref{basic1}) can be geometrized, i.e.,  that all the terms in the exact sequence \eqref{basic2} can be turned into presheaves of topological
$\Q_p$-vector spaces on the category of perfectoid affinoid spaces over $C$ (TVS for short).
 This geometrization is also essential for  the study of dualities (see Section \ref{poin4}).
\subsection{BC's, qBC's,  and TVS's}\label{geo2}
Basic examples of TVS's are:

$\bullet$ ${\Q_p}$, sending a perfectoid space $S$ to ${\cal C}(\pi_0(S),\Q_p)$,

$\bullet$ ${\mathbb G}_a$, sending $S$ to $\so(S)$.

\subsubsection{BC's}\label{geo2.1}
We let BC be the smallest abelian subcategory of TVS stable under extensions and containing
${\Q_p}$ and ${\mathbb G}_a$ (this definition is taken from~\cite{lebras}, where
it is proven that it is equivalent to the original definition~\cite{CB,CF}).  
Typical objects in this category are constructed from
the period sheaves $\Bst^+$ and $\Bdr^+$: for example, $\Bdr^+/t^m$ is a BC. 
More general examples include:

$\bullet$
$X^r_{\rm dR}(M):=(\Bdr^+\otimes_K M)/{\rm Fil}^r$,
 if $M$ is a finite rank filtered $K$-module and $r\in\Z$ (this is a successive
extensions of ${\mathbb G}_a$'s),

$\bullet$ $X^r_{\rm st}(M):=(\Bst^+\otimes_{C^{\nr}} M)^{N=0,\varphi=p^r}$  
if $M$ is a finite rank $(\varphi,N,G_K)$-module over $C^{\rm nr}$, and $r\in\Z$.

\vskip2mm
There are two functions on the category of BC's, 
additive in exact sequences:  the {\it dimension} $\dim$ and  the {\it height} ${\rm ht}$,
characterized by the fact that $\dim{\mathbb G}_a=1$ and
$\dim {\Q_p}=0$, whereas ${\rm ht}\,{\mathbb G}_a=0$ and
${\rm ht}\, {\Q_p}=1$. We define {\it the Dimension} ${\rm Dim}({\mathbb W})$ of a BC ${\mathbb W}$
as the couple $({\rm dim}({\mathbb W}),{\rm ht}({\mathbb W}))$.

  For example we have (see~\cite[Ex.\,5.14]{CN1}, where the reader will find a (more complicated) 
formula with no restriction on $r$):

$\bullet$ ${\rm Dim}(X^r_{\rm dR}(M))=(r\,{\rm rk}(M)-t_H(M),0)$ 
if $r$ is~$\geq$ all the jumps
in the filtration, and $t_H(M)=\sum_{i\in \Z} i\,{\rm rk}({\rm gr}^iM)$
(the endpoint of the Hodge polygon of $M$),

$\bullet$ ${\rm Dim}(X^r_{\rm st}(M))=(r\,{\rm rk}(M)-t_N(M),{\rm rk}(M))$
if $r\geq$ all the valuations of eigenvalues of $\varphi$ and $t_N(M)$ is the sum
of these valuations (the endpoint of the Newton polygon of $M$).

\subsubsection{qBC's}
 If $W$ is a topological\footnote{We are purposefully vague here concerning the meaning of "topological". Let us just mention that, for the theory of qBC's to work well, one needs to put 
a restriction of the type of topological spaces allowed.}
 $C$-vector space, 
$S\mapsto \so(S)\otimes^{\Box}_{C} W$ defines a TVS 
${\mathbb G}_a\otimes^{\Box}_{C} W$.
More generally, if $W$ is a topological $\bdr^+/t^m$-module, 
$S\mapsto (\Bdr^+(S)/t^m)\otimes^{\Box}_{\bdr^+/t^m} W$ defines a TVS 
$(\Bdr^+/t^m)\otimes^{\Box}_{\bdr^+/t^m} W$.

We define a Topological $\Bdr^+$-Module to be a TVS of this form.
And we define~\cite{CN5} a qBC to be a TVS having a finite filtration whose
associated graded pieces are BC's or Topological $\Bdr^+$-Modules.

The height of a qBC is still defined and additive in exact sequences, 
but its dimension can be infinite.

\subsubsection{TVS's}
To get the right $\Ext$-groups between BC's in the TVS category (see Section \ref{poin2}) the category TVS needs to appropriately modified (see \cite{TVS}).
 For one, one needs to restrict the test objects to strictly totally disconnected spaces.
 But also all the presheaves need to be topologically enriched\footnote{Actually this is already the case for the qBC's but we ignored it.}, which means that the mapping spaces between affinoid perfectoids are equipped with a topology and the presheaf functors are sensitive to this topology. This forces, for example, a topological Yoneda Lemma. It also 
 rigidifies the topological sheaves enough for them to reflect well the category of BC's, that is, the canonical functor from BC's to TVS's is fully faithful (in the derived sense).

\subsection{Geometrization}\label{geo3}
The rings of periods are naturally TVS's $S\mapsto \Bdr^+(S)$ and $S\mapsto\Bst^+(S)$,
which gives rise to the TVS's ${\mathbb X}^{r,i}$ and ${\mathbb{DR}}^{r,i}$.
To geometrize $H^n_{\rm proet}(Y_C,\Q_p(r))$ we just have to consider\footnote{\label{idea}The idea
is that $H^n_{\rm proet}(Y_S,\Q_p(r))$ ``should'' not depend on $S$ (at least if $\pi_1(S)$
is trivial; that would indeed be the case over the complex numbers). This is of course already false
for the open unit disk as we have mentioned but the content of Theorem~\ref{geo4} 
and Proposition~\ref{geo5} is that there is a constant ``reasonable'' object that controls
$H^n_{\rm proet}(Y_S,\Q_p(r))$ when $S$ varies.}
the presheaf $\mathbb{H}^n_{\rm proet}(Y_C,\Q_p(r)): S\mapsto H^n_{\rm proet}(Y_S,\Q_p(r))$, where the pro-\'etale cohomology is equipped with its canonical topology.
\begin{theorem}\label{geo4}
{\rm (Geometrization of the basic comparison theorem~\cite{CN4})} Let $r\geq 0$. Let $Y$ be a smooth partially proper rigid analytic space over $K$. 
We have an exact sequence of TVS's (i.e.,  the sequence is exact, when evaluated on any
perfectoid affinoid space over $C$)
\begin{equation}\label{basic3}
\cdots\to {\mathbb X}^{r,n-1}\to {\mathbb{DR}}^{r,n-1}
\to {\mathbb H}^n_{\rm proet}(Y_C,\Q_p(r))\to
{\mathbb X}^{r,n}\to {\mathbb{DR}}^{r,n}\to\cdots
\end{equation}
\end{theorem}
The TVS's appearing in the sequence \eqref{basic3} have special properties;
indeed, we have (see~\cite[Cor.\,7.9]{CN5} for ${\mathbb H}^n_{\rm proet}$;
the other results are immediate): 
\begin{proposition} \label{geo5}
If $Y$ is quasi-compact or, more generally, if the de Rham cohomology of $Y$
is finite dimensional, then the exact sequence \eqref{basic3} is an
exact sequence of qBC's in which the ${\mathbb X}^{r,i}$ are BC's and the
${\mathbb{DR}}^{r,i}$'s are Topological $\Bdr^+$-Modules. 
\end{proposition}

\begin{proposition} \label{geo6}
{\rm(\cite[Prop.\,1.17]{CN5})}
If $Y$ is small {\rm(meaning that its de Rham cohomology is finite dimensional)}, then the 
${\rm C}_{\rm st}$-conjecture for $H^n_{\proeet}(Y_C,\Q_p(r))$ 
is equivalent to each of the following statements:

{\rm (i)} ${\rm ht}({\mathbb H}^n_{\proeet}(Y_C,\Q_p(r)))=\dim_K H^n_{\rm dR}(Y)$,

{\rm (ii)} $H^1(X_{\rm FF},{\cal E}(H^i_{\rm HK}(Y_C),H^i_{\rm dR}(Y)))=0$, for $i=n-1,n$,
where ${\cal E}(H^i_{\rm HK}(Y_C),H^i_{\rm dR}(Y))$ 
is the vector bundle of the Fargues-Fontaine curve $X_{\rm FF}$
associated to the $\varphi$-module $H^i_{\rm HK}(Y_C)$ modified at~$\infty$ by the
lattice corresponding to ${\rm Fil}^0(\bdr^+\otimes^{\Box}_K H^i_{\rm dR}(Y))$.
\end{proposition}
Criterium (i) is a bit surprising, 
considering how big  ${\mathbb H}^n_{\proeet}(Y_C,\Q_p(r))$ is in general.
Criterium (ii) is very useful for checking cases of the ${\rm C}_{\rm st}$-conjecture.

\section{Poincar\'e duality}\label{poin1}
As we mentioned in Section~\ref{sad1}, the existence of a Poincar\'e duality for 
geometric $p$-adic pro-\'etale cohomology of nonproper
analytic varieties looked rather problematic. Still,  the computations
we have done, assuming that one can extract
the arithmetic cohomology from the geometric one using Hochschild-Serre spectral sequence,
pointed strongly towards the existence of a usual (topological)
duality for arithmetic $p$-adic pro-\'etale cohomology.
It made us look harder for a duality in the geometric case.  The key obstacle was to make
``the $\Q_p$-dual of $C$'' be $C$, which is clearly not possible.
But we realized that, interpreting
$C$ as the $C$ points of ${\mathbb G}_a$ (a natural thing to do, considering
Theorem~\ref{geo4}), we could interpret the $\Q_p$-dual of $C$ as being
${\rm Ext}^1_{\rm BC}({\mathbb G}_a,\Q_p)$ which is indeed equal to $C$
(this computation is a crucial ingredient in the
early theory of BC's~\cite[Prop.\,9.16]{CB}). This led us to Conjecture~\ref{poin5} below~\cite{COB2}.

\subsection{$\Ext$-groups} \label{poin2} 
The internal
$\Ext$-groups between BC's, in the TVS category, can be computed using the following results:
\begin{align}\label{slonce2}
&{{\Hhom}}_{\rm TVS}({\Q_p},{\Q_p})\simeq {\Q_p},
&&{{\Hhom}}_{\rm TVS}({\Q_p},{\mathbb G}_a)\simeq {\mathbb G}_a,\\ \notag
&{{\Hhom}}_{\rm TVS}({\mathbb G}_a,{\Q_p})= 0,
&&{\Hhom}_{\rm TVS}({\mathbb G}_a,{\mathbb G}_a)\simeq {\mathbb G}_a,\\ \notag
&{{\mathcal Ext}}^1_{\rm TVS}({\Q_p},{\Q_p})=0,
&&{{\mathcal Ext}}^1_{\rm TVS}({\Q_p},{\mathbb G}_a)=0,\\ \notag
&{{\mathcal Ext}}^1_{\rm TVS}({\mathbb G}_a,{\Q_p})\simeq {\mathbb G}_a,
&&{{\mathcal Ext}}^1_{\rm TVS}({\mathbb G}_a,{\mathbb G}_a)\simeq {\mathbb G}_a. 
\end{align}
And ${{\mathcal Ext}}^i_{\rm TVS}(W_2,W_1)=0$, for all $i\geq 2$ and all BC's $W_1,W_2$.
The nontrivial ${{\mathcal Ext}}^1$'s above are generated by the classes of the extensions
\begin{align*}
 & 0\to\Q_p t\to(\Bcris^+)^{\varphi=p}\to{\mathbb G}_a\to 0,\\
 & 0\to {\mathbb G}_a t\to \Bdr^+/t^2\to {\mathbb G}_a\to 0
 \end{align*}
\begin{remark}
\begin{enumerate}[leftmargin=*]
\item The above computations are the same when done  in the BC-category but the vanishing
of ${\mathcal Ext}^i$'s in the BC-category, for $i\geq 2$, is elementary contrary to 
the vanishing in the TVS-category.
\item The computation of  $\Hhom$'s and ${\mathcal Ext}^1$'s in \eqref{slonce2} was already done in \cite{CB}. The vanishing of the higher ${\mathcal Ext}$'s was proved by Ansch\"utz-Le Bras in \cite{ALB} in the category of pro-\'etale $\Q_p$-sheaves (with no objectwise topology) by a reduction, via MacLane-Breen  resolutions, to an analogous vanishing result  of Breen in characteristic $p$. 
The vanishing in the TVS-category is deduced in \cite{TVS}  from that via a fully-faithfulness result (for the functor from the category of pro-\'etale $\Q_p$-sheaves  to TVS's). 
\end{enumerate}
\end{remark}
\subsection{Geometric Duality}\label{poin4}   If $Y$ is a smooth proper rigid analytic variety over $K$,  Mann \cite{Mann} and Zavyalov \cite{Zav} have proved that the $p$-adic pro-\'etale cohomology of $Y_C$ satisfies Poincar\'e duality akin to the one for algebraic varieties. We note that in this case the cohomology groups have finite rank over $\Q_p$. We know now that we also have a Poincar\'e duality for smooth partially proper rigid analytic varieties over $K$ but now: (1) the cohomology groups are infinite dimentional, (2) there are, in general,  
nontrivial $\Ext$-groups hence we get a genuine Verdier-type duality.

\subsubsection{Heuristic: The $p$-adic unit disk} \label{poin6} The initial evidence that Poincar\'e duality should hold (and its precise shape) came from the following heuristic computations. 

 Let $Y$ be the open unit disk over $K$. From the fundamental diagram (see Theorem \ref{affinoids} and Corollary \ref{affinoids-comp}) we get the  exact sequence of solid $\Q_p$-modules
\begin{align*}
& 0\to \so(\partial Y_C)/\so(Y_C)\to H^2_{\proeet,c}(Y_C,\Q_p(1))\to \Q_p\to 0
\end{align*}
and isomorphisms
$$
H^i_{\proeet}(Y_{C},\Q_p)\simeq 
\begin{cases} \Q_p &{\text{if $i=0$,}}\\ \O(Y_{C})/C &{\text{if $i=1$,}}\\
0 &{\text{if $i\geq 2$.}}\end{cases}
$$
Since the groups $\O(Y_{C})/C \simeq H^0(Y_{C},\Omega^1)$ (via $f\mapsto df$)
 and $\O(\partial Y_{C})/\O(Y_{C})\simeq H^1_c(Y_{C},\O)$ are in
Serre duality over $C$, 
 the above computation of $\Ext$-groups in the TVS-category suggests that 
\begin{equation}\label{sad21}
H^1_{\proeet}(Y_C,\Q_p)\simeq 
\underline{\rm Ext}^1_{\rm TVS}(\mathbb{H}^2_{\proeet,c}(Y_C,\Q_p(1)),\Q_p).
\end{equation}
Similarly, one guesses that we should have an exact sequence ($\underline{\Ext}^i$ is the $\Ext^i$ in the TVS-category)
\begin{equation}\label{sad2}
0\to \underline{\rm Ext}^1_{\rm TVS}(\mathbb{H}^1_{\proeet}(Y_C,\Q_p),\Q_p)\to
H^2_{\proeet,c}(Y_C,\Q_p(1))\to \underline{\Hom}_{\rm TVS}(\mathbb{H}^0_{\proeet}(Y_C,\Q_p),\Q_p)\to 0
\end{equation}
This is because we should  have
\begin{align*}
\so(\partial Y_C)/\so(Y_C) & \simeq \underline{\Ext}^1_{\rm TVS}(\mathbb{H}^1_{\proeet}(Y_C,\Q_p),\Q_p),\\
\Q_p& \simeq \underline{\Hom}_{\rm TVS}(\mathbb{H}^0_{\proeet}(Y_C,\Q_p),\Q_p).
\end{align*}
\subsubsection{Statement}  Let $Y$ be a smooth, partially proper rigid analytic variety over $K$ of dimension $d$. Let  ${\mathbb R}_{\proeet,*}(Y_C,\Q_p(r))$ be  the TVS's representing  cohomology $\R\Gamma_{\proeet,*}(Y_C,\Q_p(r))$, i.e., we have 
$${\mathbb R}_{\proeet.*}(Y_C,\Q_p(r))({\rm Spa}(C))\simeq \R\Gamma_{\proeet,*}(Y_C,\Q_p(r)).
$$
The above computations led to the following conjecture: 
\begin{conjecture}\label{poin5}
{\rm(Geometric duality, \cite[Th. 1.7]{CGN1})} There is a natural duality quasi-isomorphism in TVS-category
\begin{equation}\label{dual1}
{\mathbb R}_{\proeet}(Y_C,\Q_p)\simeq \R\Hhom_{\rm TVS}({\mathbb R}_{\proeet,c}(Y_C,\Q_p(d))[2d],\Q_p).
\end{equation}
\end{conjecture}
This conjecture is now a theorem. There are two different proofs: by Ansch\"utz-Le Bras-Mann \cite{ALBM} and Colmez-Gilles-Nizio{\l} \cite{CGN1} (more precisely, the former replaces TVS's with pro-\'etale sheaves of $\Q_p$-modules).  We will briefly discuss these proofs in Section \ref{CGN}.
\begin{remark}\label{poin7}
\begin{enumerate}[leftmargin=*]
\item So far we do not have a setting in which a generalization of the duality \eqref{sad2} would hold. It fails in the condensed setting since there are too many extensions between Banach spaces.  We phantasize that it could work in the qBC-category. 
\item  Computations suggest that there should be a Poincar\'e  duality for the boundary $\partial Y$ of any smooth partially proper variety $Y$  over $K$. But, again, it is not clear in which setting this should hold. This would follow from the duality \eqref{dual1} if we had the duality "the other way" from (1). Let us state it: 
\begin{conjecture}{\rm (Duality for boundary pro-\'etale cohomology)}
 There is a natural  quasi-isomorphism in $\sd(?)$
\begin{align*}
 {\mathbb R}_{\proeet}(\partial Y_C,\Q_p)\stackrel{\sim}{\to} \R\Hhom_{?}({\mathbb R}_{\proeet}(\partial Y_C,\Q_p(d))[2d-1],\Q_p).
 \end{align*}
\end{conjecture}
\begin{proof}{\rm (Heuristics)}
Use the distinguished triangle \eqref{basic2} and dualize it. Then use the duality \eqref{dual1} and its hypothetical "inverse".
\end{proof}
This can be proved   in the arithmetic setting in dimension $1$ (see below).
\end{enumerate}
\end{remark}

 The following  pleasing corollary  is implied by   the duality \eqref{dual1} and the vanishing of TVS-$\Ext$-groups in dimensions $\geq 2$:
\begin{corollary}  {\rm (Verdier exact sequence, \cite[Cor.\,1.8]{CGN1}} Let $Y$ be a smooth Stein variety over $K$ of dimension $d$. Let $i\geq 0$. 
There exists a short exact sequence in {\rm TVS}'s
$$
0\to {\mathcal Ext}^1_{\rm TVS}({\mathbb H}^{2d-i+1}_{\proeet,c}(Y_C,\Q_p(d)),\Q_p)\to {\mathbb H}^i_{\proeet}(Y_C,\Q_p)\to{\mathcal Hom}_{\rm TVS}({\mathbb H}^{2d-i}_{\proeet,c}(Y_C,\Q_p(d)),\Q_p)\to 0
$$
\end{corollary}

\subsubsection{Proofs of Conjecture \ref{poin5}}\label{CGN}

 We start with the proof of the  geometric Poincar\'e duality from \cite{CGN1}.
  Just as in Section \ref{poin6}, via the basic comparison theorem (see Theorem \ref{basic1} and its compactly supported analog) the duality for $p$-adic pro-\'etale cohomology can be reduced to the one for Hyodo-Kato  and filtered de Rham cohomologies.
 The first of those, in turn, can be reduced (as in \cite{AGN}), via the Hyodo-Kato isomorphism, to the second one (which follows from coherent Serre duality).
 That should have been enough to give a proof of pro-\'etale duality however, at the moment, the argument goes through the Fargues-Fontaine curve that geometrically separated the Hyodo-Kato and de Rham parts (something that is more difficult to do directly in TVS's) and makes the bootstrapping  of the de Rham-dualities up to pro-\'etale cohomology easy.
 The weight of the argument then moves to controlling  the descent from the Fargues-Fontaine curve to TVS's, a nontrivial task. 

    The proof in \cite{ALBM} also proceeds  in two steps.
 The duality on the Fargues-Fontaine curve is obtained from a 6-functor
 formalism for solid quasi-coherent sheaves.
 The descent is now not to TVS's  but to the algebraic category of pro-\'etale $\Q_p$-sheaves.
 Like all the descents we have so far in this topological setting it does not work perfectly and for it to work here one needs to know functional analytic properties of the solid quasi-coherent sheaves involved in the duality theorem.
 That is derived from the comparison theorems for period sheaves (see   Theorem \ref{CGN21} and  Remark \ref{BoscHK}). 

  The  method from \cite{ALBM} allows to prove more general dualities than the ones stated in Conjecture \ref{poin5}.
 For example, it yields a Poincar\'e duality for smooth proper rigid analytic spaces and $\Q_p$-local systems.
 This case is interesting because the pro-\'etale cohomology groups are BC's but not finite rank over $\Q_p$, in general.
 Another approach to this result was developed recently by Li-Nizio{\l}-Reinecke-Zavyalov \cite{LNRZ}.

\subsection{Arithmetic duality}\label{arit2}
Arithmetic duality for smooth partially proper analytic spaces was easier  
to understand than its geometric counterpart because it did not involve nontrivial $\Ext$-groups. 
Moreover, in the Stein case there it has a nonderived version (both ways). 

In this section, $K$ is a finite extension of $\Q_p$.

\begin{theorem}\label{poin8}
{\rm(Arithmetic duality, \cite[Th. 1.1]{CGN}, \cite[Th. 1.1]{ZL})}  Let $Y$ be a smooth Stein  rigid analytic variety over $K$ of dimension $d$.
There are natural duality isomorphisms 
\begin{align*}
  H^{i}_{\proeet,c}(Y,\Q_p) & \stackrel{\sim}{\to} \underline{\Hom}(H^{2d+2-i}_{\proeet}(Y,\Q_p(d+1)),\Q_p),\\
  H^{i}_{\proeet}(Y,\Q_p) & \stackrel{\sim}{\to} \underline{\Hom}(H^{2d+2-i}_{\proeet,c}(Y,\Q_p(d+1)),\Q_p).
\end{align*}
\end{theorem}
A key result  in the proof of arithmetic duality for analytic curves in \cite{CGN} is a Poincar\'e duality for the ghost circle\footnote{Here the pro-\'etale cohomology needs to be defined appropriately.}. 
\begin{theorem}{\rm (Arithmetic duality for ghost circle)}\label{main-arithmeticY0} Let $D$ be an     open disk over $K$.  Let $Y:=\partial D$
be the ``ghost circle''.
 There is a natural duality  isomorphism of solid $\Q_p$-vector spaces
\begin{align*}
\quad H^i_{\proeet}(Y,\Q_p(j))\stackrel{\sim}{\to} H^{3-i}_{\proeet}(Y,\Q_p(2-j))^*.
\end{align*}
\end{theorem}
Hence $\partial D_C$ behaves like a proper variety of dimension $\frac{1}{2}$ (since 
the ``arithmetic variety'' $\partial D$
behaves like a proper variety of dimension~$\frac{3}{2}$), hence
like a proper variety of "real" dimension $1$ (as in the archimedean world!).

\end{document}